\RequirePackage{fix-cm}
\documentclass[smallextended]{svjour3}       
\smartqed  
\usepackage{graphicx}
\usepackage{bm}
\usepackage{amssymb}
\usepackage{amsmath,amsfonts}
\usepackage{latexsym}
\usepackage{mathrsfs}
\usepackage{color}
\usepackage{subfigure}
\usepackage{dsfont}
\usepackage{hyperref}  
\usepackage{stmaryrd}
\usepackage[linesnumbered,ruled]{algorithm2e}
\usepackage[flushleft]{threeparttable}
\usepackage{comment}

\newtheorem{defn}{Definition}

\newcommand{\be}{\begin{equation}}
\newcommand{\ee}{\end{equation}}

\DeclareMathOperator*{\Argmax}{Argmax}  
\DeclareMathOperator*{\Argmin}{Argmin}  
\DeclareMathOperator*{\argmin}{argmin}  

\def\diag{\mbox{diag}}

\def\bA{\bm{A}}
\def\bE{\bm{E}}
\def\bI{\bm{I}}
\def\bK{\bm{K}}
\def\bP{\bm{P}}
\def\bQ{\bm{Q}}

\def\bR{\bm{R}}

\def\bU{\bm{U}}
\def\bV{\bm{V}}
\def\bW{\bm{W}}

\def\bY{\bm{Y}}

\def\bLbd{\bm{\Lambda}}

\def\tbR{\tilde{\bm{R}}}

\def\bh{\bm{h}}

\def\b0{\bm{0}}

\def\bc{\bm{c}}
\def\bcp{\bm{c}^\prime}

\def\bee{\bm{e}}
\def\bk{\bm{k}}
\def\bo{\bm{1}}

\def\bq{\bm{q}}
\def\bu{\bm{u}}
\def\bv{\bm{v}}
\def\bvp{\bm{v}^\prime}

\def\bx{\bm{x}}
\def\bxs{\bm{x}^*}
\def\by{\bm{y}}
\def\bz{\bm{z}}
\def\bD{\bm{D}}
\def\bH{\bm{H}}
\def\he{\hat{\bm{e}}}

\def\E{\mathbb{E}}
\def\R{\mathbb{R}}
\def\mcF{\mathcal{F}}
\def\mF{\mathcal{F}}
\def\mO{\mathcal{O}}
\def\mP{\mathcal{P}}
\def\P{\mathbb{P}}
\def\mcI{\mathcal{I}}
\def\mcJ{\mathcal{J}}

\begin{document}

\title{Non-Convex Exact Community Recovery in Stochastic Block Model \thanks{A preliminary version of this work has appeared in the Proceedings of the 37th International Conference on
Machine Learning (ICML 2020), 2020 \cite{wang2020non}. The first and third authors are supported in part by the Hong Kong Research Grants Council (RGC) General Research Fund (GRF) project CUHK 14208117 and in part by the CUHK Research Sustainability of Major RGC Funding Schemes project 3133236. The second author is supported in part by the National Natural Science Foundation of China (NSFC) project 11901490 and in part by a HKBU Start-up Grant. Most of the work of the second author was done when he was affiliated with the Department of Mathematics of the Hong Kong Baptist University.}
}

\titlerunning{Non-Convex Exact Community Recovery}        

\author{Peng Wang  \and
        Zirui Zhou \and \\ Anthony Man-Cho So 
}

\authorrunning{P. Wang, Z. Zhou, A. M.-C. So} 

\institute{Peng Wang \at
              Department of Systems Engineering and Engineering Management \\ The Chinese University of Hong Kong, Shatin, NT, Hong Kong \\
              \email{wangpeng@se.cuhk.edu.hk}           
           \and
           Zirui Zhou \at
           Huawei Technologies Canada Co., Ltd., Burnaby, Canada \\
           \email{zirui.zhou@huawei.com}
           \and 
           Anthony Man-Cho So \at
              Department of Systems Engineering and Engineering Management \\ The Chinese University of Hong Kong, Shatin, NT, Hong Kong \\
              \email{manchoso@se.cuhk.edu.hk} 
}

\date{Received: date / Accepted: date}

\maketitle

\begin{abstract}
Community detection in graphs that are generated according to stochastic block models (SBMs) has received much attention lately. In this paper, we focus on the binary symmetric SBM---in which a graph of $n$ vertices is randomly generated by first partitioning the vertices into two equal-sized communities and then connecting each pair of vertices with probability that depends on their community memberships---and study the associated exact community recovery problem.
Although the maximum-likelihood formulation of the problem is non-convex and discrete, we propose to tackle it using a popular iterative method called projected power iterations. To ensure fast convergence of the method, we initialize it using a point that is generated by another iterative method called orthogonal iterations, which is a classic method for computing invariant subspaces of a symmetric matrix. We show that in the logarithmic sparsity regime of the problem, with high probability the proposed two-stage method can exactly recover the two communities down to the information-theoretic limit in $\mO(n\log^2n/\log\log n)$ time, which is competitive with a host of existing state-of-the-art methods that have the same recovery performance. We also conduct numerical experiments on both synthetic and real data sets to demonstrate the efficacy of our proposed method and complement our theoretical development. 
\end{abstract}

\section{Introduction}
\label{intro}

Community detection is a fundamental task in network analysis and has found many applications in diverse fields such as physics \cite{fortunato2010community,newman2004finding}, biology \cite{cline2007integration}, and social science \cite{girvan2002community}, to name a few. In research on community detection, the stochastic block model (SBM), which provides a way to generate graphs with community structure, is widely used as a platform for validating theoretical ideas and comparing numerical algorithms. In particular, substantial advances have been made in the past decade on understanding the fundamental limits of community detection in graphs that are generated by SBMs, and on developing computationally tractable methods that can meet different recovery requirements up to their corresponding fundamental limits; see, e.g., \cite{abbe2018community} and the references therein.

One problem that has been extensively studied in the literature is the exact recovery of communities in the binary symmetric SBM (also known as the planted bisection model). Specifically, given an $n$-vertex graph with two equal-sized hidden communities, and each pair of vertices in the graph is connected by an edge with probability $p$ if they both belong to the same community and with probability $q$ otherwise, the goal is to achieve \emph{exact recovery} (i.e., recover the underlying communities exactly with high probability) using only the adjacency matrix of the graph. It is well known that whether exact recovery is achievable depends on the scalings of $p$, $q$, and $p-q$. When $p = a/n$ and $q = b/n$ for some $a>b>0$ (the constant sparsity regime), it is impossible to recover the communities because the graph is disconnected with high probability \cite{decelle2011asymptotic}. On the other hand, when $p = \alpha\log n/n$ and $q = \beta\log n/n$ for some $\alpha>\beta>0$ (the logarithmic sparsity regime), Abbe et al. \cite{abbe2016exact} and Mossel et al. \cite{mossel2014consistency} independently showed that exact recovery is impossible if $\sqrt{\alpha} - \sqrt{\beta}<\sqrt{2}$ but is possible if $\sqrt{\alpha} - \sqrt{\beta} > \sqrt{2}$, thereby establishing a sharp threshold for exact recovery. The proof of Abbe et al.~\cite{abbe2016exact} takes an information-theoretic approach and obtains the said threshold by analyzing the following maximum-likelihood (ML) formulation of the problem:
\begin{equation}
\label{eq:MLE} \tag{{\sf MLE}}
\max \left\{\bx^T\bA\bx: \, \mathbf{1}_n^T\bm{x}=0, \, x_i \in \{\pm 1\}, \, i = 1,\dots,n \right\}.
\end{equation}
Here, $\bA$ is the adjacency matrix of the graph, $\bo_n$ is the all-one vector of dimension $n$, and $x_i\in\{\pm1\}$ encodes the community membership of vertex $i$ for $i=1,\ldots,n$. It is shown in~\cite{abbe2016exact} that when $\sqrt{\alpha}-\sqrt{\beta}<\sqrt{2}$, the ML estimator (i.e., an optimal solution of \eqref{eq:MLE}) fails to recover the communities with probability bounded away from zero for sufficiently large $n$, but when $\sqrt{\alpha}-\sqrt{\beta}>\sqrt{2}$, the ML estimator can exactly recover the communities with high probability.

From a computational point of view, solving Problem~\eqref{eq:MLE} amounts to finding a minimum bisection of a graph, which is NP-hard in the worst case~\cite{garey1974some}. Over the past few decades, many algorithms have been proposed to tackle the problem of exact community recovery in the binary symmetric SBM; see, e.g.,~\cite{abbe2018community} for a summary of some of the earlier works. In view of the information-theoretic limit established in~\cite{abbe2016exact,mossel2014consistency}, a natural task is to design efficient algorithms that can exactly recover the communities down to the limit. This has also been undertaken  in~\cite{abbe2016exact,mossel2014consistency}. The former presents a two-stage algorithm that combines the partial recovery algorithm of Massouli\'{e}~\cite{massoulie2014community} with a local improvement procedure, while the latter gives a three-stage algorithm that uses spectral clustering for initialization and then combines a partial recovery step with a local refinement procedure. The former also presents an algorithm based on a semidefinite relaxation (SDR) of Problem~\eqref{eq:MLE} and poses the conjecture that the algorithm can exactly recover the communities down to the information-theoretic limit. The conjecture was later resolved in the affirmative independently by Hajek et al.~\cite{hajek2016achieving} and Bandeira~\cite{bandeira2018random}. 

Subsequent to the above development, a variety of efficient algorithms with the same recovery performance in the binary symmetric SBM 
have appeared in the literature. For instance, Abbe and Sandon~\cite{abbe2015community} developed a two-stage algorithm that is similar in spirit to the one in~\cite{abbe2016exact}. Yun and Proutiere \cite{yun2016optimal} presented a spectral partition algorithm, which proceeds by applying spectral decomposition to a trimmed adjacency matrix, followed by some local improvements. Later, Gao et al.~\cite{gao2017achieving} proposed a two-stage algorithm that employs spectral clustering for initialization and penalized local maximum likelihood estimation for local refinement. It is worth noting that the aforementioned algorithms apply not only to the binary symmetric SBM but also to more general SBMs. More recently, Abbe et al. \cite{abbe2017entrywise} showed that the vanilla spectral method, which computes the eigenvector associated with the second largest eigenvalue of the adjacency matrix and uses the signs of the entries to identify the communities, already has the desired recovery performance.


Among the existing algorithms that can achieve exact recovery down to the information-theoretic limit in the binary symmetric SBM, the best complexity bound is nearly linear. This is attained by, e.g., the three-stage algorithm of Mossel et al.~\cite{mossel2014consistency} and the spectral partition algorithm of Yun and Proutiere \cite{yun2016optimal}, both of which have an $\mO(n\log^2n)$ runtime, and the two-stage algorithm of Abbe and Sandon~\cite{abbe2015community}, which has a runtime of $o(n^{1+\epsilon})$ for any $\epsilon>0$. It should be pointed out that even though the vanilla spectral method is conceptually much simpler than these algorithms, it needs to perform an eigenvector computation, and standard complexity analyses of the commonly used methods for this purpose (such as orthogonal iteration) only yield a quadratic bound at best (see, e.g.,~\cite[Part V]{trefethen1997numerical}).


\subsection{Our Contributions}
In this work, we propose a two-stage iterative algorithm that aims to achieve exact recovery in the binary symmetric SBM by directly tackling the non-convex ML estimation problem~\eqref{eq:MLE}. In the first stage, the algorithm applies the classic method of orthogonal iteration to compute an approximation of the eigenvector $\bm{u}_2$ associated with the second largest eigenvalue of $\bA$. Such an approximation is then used as an initialization in the second stage of the algorithm, which applies the method of projected power iteration (cf.~\cite{liu2017estimation,liu2017discrete}) to solve Problem~\eqref{eq:MLE}. The first stage is akin to that of a host of existing algorithms, such as those in~\cite{abbe2017entrywise,mossel2014consistency,yun2016optimal}. However, we are able to show that a coarse approximation of $\bm{u}_2$ is sufficient for the second stage of our proposed algorithm to find an optimal solution of Problem~\eqref{eq:MLE}, which is key to the efficiency of our algorithm. Specifically, we show that in the logarithmic sparsity regime of the binary symmetric SBM, our proposed algorithm achieves exact recovery all the way down to the information-theoretic limit within $\mO(\log n/\log\log n)$ orthogonal iterations and $\mO(\log n/\log\log n)$ projected power iterations, where each orthogonal iteration and projected power iteration can be implemented in $\mO(n\log n)$ time. This yields an overall complexity bound of $\mO(n\log^2n/\log\log n)$ for our algorithm, which is competitive with some of the most efficient algorithms in the literature that have the same recovery performance. By combining our techniques with the results in~\cite{abbe2017entrywise}, we can further show that the vanilla spectral method can be implemented in $\mO(n\log^2n/\log\log n)$ time. To the best of our knowledge, this is currently the best complexity bound for the method in the context of exact community recovery in the binary symmetric SBM. We also conduct numerical experiments on synthetic and real data sets to evaluate the performance of our proposed algorithm. The results demonstrate the efficacy of the algorithm and complement our theoretical development.

In recent years, there has been a growing body of literature exploring the design and analysis of fast methods for tackling non-convex formulations that arise in applications. These include deep neural networks~\cite{sun2020global,sun2020optimization}, low-rank matrix recovery~\cite{chi2019nonconvex,li2020nonconvex}, phase retrieval~\cite{ma2020implicit,vaswani2020nonconvex}, source localization~\cite{liu2017local,pun2020dynamic}, and synchronization~\cite{liu2017estimation,zhong2018near}. As these works show, the non-convex formulations in question often possess structures that can be exploited by simple and scalable methods, thereby allowing optimal solutions of those formulations to be found efficiently. Our work contributes to this emerging area by showing that in the logarithmic sparsity regime of the binary symmetric SBM, the ML estimation problem~\eqref{eq:MLE}, albeit non-convex and discrete, can be solved to optimality via a carefully designed, yet simple, iterative procedure. Prior to our work, Bandeira et al. \cite{bandeira2016low} considered another non-convex formulation of the community recovery problem, which is obtained by applying the Burer-Monteiro decomposition~\cite{burer2003nonlinear} to the semidefinite relaxation of Problem \eqref{eq:MLE}. They showed that all second-order stationary points of the non-convex formulation, which can be computed efficiently by the Riemannian trust-region method~\cite{boumal2018global}, correspond to the underlying communities with high probability as long as $(p-q)/\sqrt{p+q}\geq cn^{-1/6}$ for some constant $c>0$. Despite its low computational complexity, the approach requires a much stronger condition on $p$ and $q$ to ensure exact recovery. In particular, it cannot guarantee exact recovery in the logarithmic sparsity regime of the binary symmetric SBM.

Lastly, let us highlight the improvements made in this paper over its preliminary version~\cite{wang2020non}. First, the method in \cite{wang2020non} is designed for a regularized version of Problem \eqref{eq:MLE}. When applying the method to real data sets, which in general are not generated by the SBM, it is difficult to tune the regularization parameter. In the current work, we circumvent this difficulty by handling Problem \eqref{eq:MLE} directly, which makes our proposed method simpler and more practical. Second, compared to its regularized version in \cite{wang2020non}, Problem \eqref{eq:MLE} is more challenging as it contains an additional linear constraint. Nevertheless, we show that a suitably initialized projected power method can solve it efficiently and enjoys the same recovery guarantee as that in \cite{wang2020non}. Third, although both the method in~\cite{wang2020non} and the one proposed in this paper have the property that an iterate will converge in one step to an optimal solution of Problem~\eqref{eq:MLE} once the former is in a suitable neighborhood of the latter, we show in this paper that the size of the neighborhood can be as large as $\mO(\sqrt{\log n})$, which improves upon the $\mO(1)$ bound established in \cite{wang2020non}.

\subsection{Organization}
The rest of this paper is organized as follows. In Section \ref{sec:main-results}, we introduce the proposed two-stage algorithm for exact community recovery and present the main result of this paper. In Section~\ref{sec:pf-main}, we prove the main result and discuss its consequences. We then report some numerical results in Section \ref{sec:num} and conclude in Section~\ref{sec:concl}.

\emph{Notation.}  Let $\R^n$ be the $n$-dimensional Euclidean space. We write matrices in bold capital letters like $\bA$, vectors in bold lower-case letters like $\bm{a}$, and scalars in plain letters. Given a matrix $\bA \in \R^{m\times n}$, we use $\sigma_{\max}(\bA)$ or $\|\bA\|$ to denote its largest singular value (i.e., spectral norm), $\sigma_{\min}(\bA)$ its smallest singular value, and $a_{ij}$ its $(i,j)$-th element. If $\bA$ is symmetric, then we use $\lambda_{\min}(\bA)$ to denote its smallest eigenvalue. Given a vector $\bx\in\mathbb{R}^n$, we use $\|\bx\|_2$ to denote its Euclidean norm, $x_i$ its $i$-th element, and $\diag(\bx)$ the diagonal matrix with $\bx$ on its diagonal. We use $\bo_n$ and $\bE_n$ to denote the $n$-dimensional all-one vector and $n\times n$ all-one matrix, respectively, and simply write $\bo$ and $\bE$ when their dimensions can be inferred from the context. Given a positive integer $n$, we denote by $[n]$ the set $\{1,\ldots,n\}$. Given a discrete set $T$, we denote by $|T|$ the number of elements in $T$. We use $\mbox{sgn}$ to denote the element-wise sign function; i.e., for any $\bx\in\mathbb{R}^n$,
\[
[\mbox{sgn}(\bx)]_i = \left\{
\begin{array}{c@{\quad}l}
1, & \mbox{if } x_i>0, \\
0, & \mbox{if } x_i=0, \\
-1, & \mbox{if } x_i<0,
\end{array}
\right. \quad i \in [n].
\]
We use $\mathbf{Bern}(p)$ to denote the Bernoulli random variable with mean $p$. Given two random variables $X$ and $Y$, we write $X\overset{d}{=}Y$ if $X$ and $Y$ are equal in distribution.

\section{Preliminaries and Main Results}\label{sec:main-results}

In this section, we formally state the considered problem, present the proposed algorithm, and give a summary of our main theoretical results.

To begin, let us introduce a central object in our study---the binary symmetric SBM.

\begin{defn}[Binary Symmetric SBM] \label{model:SBM}
Let $n\ge2$ be an even integer and $p,q\in[0,1]$ be parameters with $p>q$. Furthermore, let $\bx^*\in\{-1,+1\}^n$ be a label vector representing a partition of $[n]$ into two equal-sized subsets (in particular, $\bo^T\bx^*=0$). We say that a random graph $G$ is generated according to the binary symmetric SBM with parameters $(n,p,q)$ and label $\bx^*$ if $G$ has vertex set $V=[n]$ and the elements $\{a_{ij}\}_{1\le i\le j\le n}$ of its adjacency matrix $\bA$ are generated independently by
	\begin{align}
	a_{ij}\sim \left\{
	\begin{aligned}
	\mathbf{Bern}(p),\quad & \text{if}\ \ x_i^*x_j^*=1,\\
	\mathbf{Bern}(q),\quad & \text{if}\ \ x_i^*x_j^*=-1.
	\end{aligned}
	\right. \label{eq:SBM}
	\end{align}
\end{defn}
Intuitively, the label vector $\bx^*$ induces two equal-sized communities in the graph $G$. Note that we allow self-loops in $G$, though our analysis also applies to the case where no self-loop is allowed (i.e., $a_{ii}=0$ for all $i\in[n]$); see Section~\ref{sec:no-loop}.

Now, given a realization of $G$ that is generated according to the binary symmetric SBM, the problem of interest is to recover the two communities. Since $-\bxs$ represents the same community structure as $\bx^*$, this is equivalent to identifying $\bx^*$ or $-\bx^*$ from the adjacency matrix $\bA$. As in~\cite{abbe2016exact}, we say that an estimator achieves \emph{exact recovery}\footnote{This is also termed \emph{strong consistency} in the literature; see~\cite{mossel2014consistency}.} if it yields $\bx^*$ or $-\bx^*$ with probability tending to one as $n\rightarrow\infty$, where the probability is taken with respect to the distribution in~\eqref{eq:SBM}.

In this paper, we focus on the logarithmic sparsity regime of the binary symmetric SBM---i.e.,
\begin{equation}	\label{eq:p-q}
	p = \frac{\alpha\log n}{n} \quad \mbox{and} \quad q = \frac{\beta\log n}{n} 
\end{equation}
for some constants $\alpha>\beta>0$---and propose to solve the community recovery problem by directly handling the non-convex ML estimation problem \eqref{eq:MLE}, even though it is NP-hard in the worst case. The first ingredient in our approach is a simple iterative procedure called the \emph{method of projected power iteration}, which is essentially the projected gradient method applied to Problem~\eqref{eq:MLE}. Specifically, let
\begin{equation}
\label{eq:def-FS}
\mF := \left\{ \bx\in\mathbb{R}^n :  \bo^T\bx = 0, \ x_i = \pm 1, \ i \in [n] \right\}
\end{equation}
denote the feasible set of Problem \eqref{eq:MLE}. Furthermore, let $\mP:\mathbb{R}^n\rightrightarrows\mathbb{R}^n$ be the projection operator onto $\mF$; i.e., for any $\bc\in\mathbb{R}^n$, 
\begin{equation}
\label{eq:def-Proj}
\mP(\bc) := \Argmin_{\bu\in\mathbb{R}^n} \left\{ \|\bu - \bc\|_2^2 : \bu\in\mF\right\}.\footnote{By convention, we use the symbol “$\Argmin$” to denote the solution set of the associated minimization problem. When the solution set is known to be a singleton, we use the symbol “$\argmin$” to denote the unique solution in the set.}
\end{equation}
Then, the projected power iterations take the form
\begin{equation}
\label{eq:PPI}
\bx^{(k)} \in \mP(\bA\bx^{(k-1)}), \quad k=1,2,\dots.
\end{equation} 
As the following proposition shows, for every $k\geq 1$, the problem of computing $\bx^{(k)}$ in \eqref{eq:PPI} boils down to that of finding the indices that correspond to the $n/2$ largest entries of $\bA\bx^{(k-1)}$, which can be done efficiently.
\begin{proposition} \label{prop:cf-proj}
For any $\bc\in\mathbb{R}^n$, it holds that $\bv\in\mP(\bc)$ if and only if 
\begin{equation}
\label{eq:bv-form}
v_\ell = \left\{
\begin{array}{rl}
1, & \quad \ell\in\mcI, \\
-1, & \quad \ell\in[n]\setminus\mcI,
\end{array}\right.
\end{equation}
where $\mcI\subset[n]$ satisfies $|\mcI| = n/2$ and $c_i\geq c_j$ for all $i\in\mcI$ and $j\in[n]\setminus\mcI$. 
\end{proposition}
\begin{proof}
By \eqref{eq:def-Proj}, we have 
\[ 
\mP(\bc) = \Argmin_{\bu\in\mathbb{R}^n} \left\{ \|\bu - \bc\|_2^2 : \bu\in\mF\right\} = \Argmax_{\bu\in\mathbb{R}^n} \left\{ \bc^T\bu : \bu\in\mF\right\}, 
\] 
where the second equality is due to the fact that $\|\bu\|_2^2 = n$ for all $\bu\in\mF$. The desired formula~\eqref{eq:bv-form} follows immediately.
\end{proof}

\smallskip
Despite its simplicity, the method of projected power iteration may not be effective for solving Problem \eqref{eq:MLE} unless a proper initial point $\bx^{(0)}$ is available. Thus, we need an additional iterative procedure, which is usually referred to as the \emph{method of orthogonal iteration} (see, e.g., \cite{golub2012matrix}) and constitutes the second ingredient in our approach, to obtain a good initial point for the projected power iterations. The method of orthogonal iteration starts with a matrix $\bQ^{(0)}\in\mathbb{R}^{n\times 2}$ with orthonormal columns. 
In iteration $k\ge1$, it computes the QR decomposition of $\bA\bQ^{(k-1)}$; i.e.,
\[   \bA\bQ^{(k-1)} = {\bQ}^{(k)}\bR^{(k)},\]
where ${\bQ}^{(k)} \in \mathbb{R}^{n\times 2}$ has orthonormal columns and $\bR^{(k)}\in\R^{2\times 2}$ is upper triangular. It is known that the distance between the subspace spanned by the columns of $\bQ^{(k)}$ and the invariant subspace of $\bA$ that corresponds to its first two dominant eigenvalues converges linearly to $0$ as $k\rightarrow\infty$; see, e.g., \cite[Theorem 8.2.2]{golub2012matrix}. For our purpose, we only perform $N$ orthogonal iterations, where $N$ is an input parameter of the algorithm. Then, we apply Ritz acceleration (see, e.g., \cite[Chapter~8.3.7]{golub2012matrix}) to the last iterate, which amounts to computing the eigenvalue decomposition of $\bQ^{(N)^T}\bA{\bQ}^{(N)}$; i.e.,
\[ \bQ^{(N)^T}\bA{\bQ}^{(N)} = \bH^{(N)}\bD^{(N)}\bH^{(N)^T}, \]
where $\bH^{(N)}\in\mathbb{R}^{2\times 2}$ is orthogonal and $\bD^{(N)} = \mbox{diag}(d_1^{(N)},d_2^{(N)})$ is diagonal with $|d_1^{(N)}| \geq |d_2^{(N)}|$. Finally, we extract the column of $\bar{\bQ}^{(N)}={\bQ}^{(N)}\bH^{(N)}$ that corresponds to the smaller eigenvalue of $\bQ^{(N)^T}\bA{\bQ}^{(N)}$ to construct a suitable initial point for the projected power iterations.


We now summarize our proposed method for solving Problem~\eqref{eq:MLE} in Algorithm \ref{alg:TSA}. It starts with a matrix $\bY\in\mathbb{R}^{n\times 2}$, whose entries are generated independently and identically from the standard normal distribution. In the first stage (lines \ref{line:oi-start}--\ref{line:oi-end} of Algorithm \ref{alg:TSA}), the algorithm performs $N$ orthogonal iterations with the initial iterate $\bQ^{(0)}$, which is obtained by orthonormalizing the columns of $\bY$ via $\bQ^{(0)} = \bY(\bY^T\bY)^{-1/2}$, and then applies the Ritz acceleration. In the second stage (lines \ref{line:ppi-start}--\ref{line:ppi-end} of Algorithm~\ref{alg:TSA}), the algorithm employs projected power iterations to refine the initial iterate $\bx^{(0)}$, which is constructed from a suitable column of $\bar{\bQ}^{(N)}$. The algorithm terminates when $\bx^{(k)} = \bx^{{(k-1)}}$ for some $k$ in the second stage, at which point it outputs $\bx^{(k)}$.

\begin{algorithm}[t]
\DontPrintSemicolon
\SetKwInput{KwInput}{Input}               
\SetKwInput{KwOutput}{Output} 
\KwInput{adjacency matrix $\bA$, positive integer $N$}
\KwOutput{label vector $\hat{\bx}$}       
choose a matrix $\bY \in \R^{n\times 2}$, whose entries are generated independently and identically from the standard normal distribution \\
\tcc{stage 1: method of orthogonal iteration with Ritz acceleration}
set $\bQ^{(0)} \leftarrow \bY(\bY^T\bY)^{-1/2}$ \label{line:oi-start} \\
\For{$k=1,2,\dots,N$}{
set $\bm{Z}^{(k)}\leftarrow \bA\bQ^{(k-1)}$ \\
 compute the QR decomposition $\bm{Z}^{(k)} = {\bQ}^{(k)}\bR^{(k)}$ \label{line:qr} \\
}
set $\bm{S}^{(N)}\leftarrow \bQ^{{(N)}^T} \bA {\bQ}^{(N)}$ \\
compute the eigen-decomposition $\bm{S}^{(N)} = \bH^{(N)}\bD^{(N)}{\bH^{(N)^T}}$, where $\bD^{(N)} = \mbox{diag}(d_1^{(N)},d_2^{(N)})$ such that $|d_1^{(N)}| \geq |d_2^{(N)}|$ \label{line:eigen} \\ \smallskip
set $\bar{\bQ}^{(N)}\leftarrow {\bQ}^{(N)}\bH^{(N)}$ \label{line:oi-end}\\ 
set $\tilde{\by}$ to be the $i^*$-th column of $\bar{\bQ}^{(N)}$, where $i^*=\argmin_{i\in\{1,2\}} d_i^{(N)}$ \label{line:evec-1} \\
set $\by\leftarrow\tilde{\by} - (\bo^T\tilde{\by}/n)\bo$ \label{line:evec-2}  \\
\tcc{stage 2: method of projected power iteration}
set $\bx^{(0)} \leftarrow \sqrt{n}\by/\|\by\|_2$ \label{line:ppi-start} \\
\For{$k=1,2,\dots$}{
set $\bz^{(k)}\leftarrow \bA\bx^{(k-1)}$ \\
set $\bx^{(k)} \leftarrow \mP(\bz^{(k)})$ \label{line:proj0} \\
\If{$\bx^{(k)}=\bx^{(k-1)}$}{
terminate and output $\hat{\bx} = \bx^{(k)}$
} 
}\label{line:ppi-end}
\caption{A Two-Stage Algorithm for Solving Problem \eqref{eq:MLE}}
\label{alg:TSA}
\end{algorithm}

We next present the main result of this paper, which shows that Algorithm~\ref{alg:TSA} achieves exact recovery at the information-theoretic limit and also provides explicit iteration complexity bounds for Algorithm \ref{alg:TSA}.

\begin{theorem}\label{thm:main}
Let $\bA$ be the adjacency matrix of a random graph generated according to the binary symmetric SBM with parameters $(n,p,q)$ and label $\bx^*$, where $p,q$ satisfy~\eqref{eq:p-q} for some constants $\alpha>\beta>0$. Set $N=\Theta(\log n/\log\log n)$. If $\sqrt{\alpha} - \sqrt{\beta} > \sqrt{2}$, then for all sufficiently large $n$, the following statement holds with probability at least $1-n^{-\Omega(1)}$: Algorithm \ref{alg:TSA} takes $\mO(\log n/\log\log n)$ orthogonal iterations and $\mO(\log n/\log\log n)$ projected power iterations to output $\bx^*$ or $-\bxs$. Here, the probability is taken with respect to the random choices in $\bA$ and in Algorithm~\ref{alg:TSA}.
\end{theorem}

We remark that the value of $N$ in Theorem \ref{thm:main} can be explicitly given; see~\eqref{eq:N1}. Equipped with Theorem \ref{thm:main}, it is not hard to derive the total computational cost of Algorithm \ref{alg:TSA}. Indeed, since $\bm{Z}^{(k)} \in \mathbb{R}^{n\times 2}$ and $\bm{S}^{(N)} \in \mathbb{R}^{2\times 2}$, the QR decomposition in line~\ref{line:qr} can be found in $\mO(n)$ time, while the eigen-decomposition in line~\ref{line:eigen} can be found in $\mO(1)$ time~\cite{trefethen1997numerical}. Moreover, by Proposition \ref{prop:cf-proj}, the projection $\mP(\bz^{(k)})$ in line~\ref{line:proj0} can be found by first identifying the $(n/2)$-th largest element $\bar{z}^{(k)}$ of $\bm{z}^{(k)}$, which can be done in $\mathcal{O}(n)$ time~\cite{blum1973time}, and then comparing each element of $\bm{z}^{(k)}$ with $\bar{z}^{(k)}$, which can be trivially done in $\mathcal{O}(n) $ time. Now, the remaining dominant computational cost is that of computing matrix-vector products of the form $\bA\bv$, where $\bA\in\mathbb{R}^{n\times n}$ is generated according to the setting of Theorem~\ref{thm:main} and $\bv\in\mathbb{R}^n$ is arbitrary. Using a simple concentration argument, one can show that the number of non-zero entries in $\bA$ is $\mO(n\log n)$ with high probability; see Section~\ref{sec:pf-main}. Hence, with high probability, the cost of computing $\bA\bv$ is $\mO(n\log n)$ for any $\bv$. Putting the above time bounds together and using the iteration bounds established in Theorem~\ref{thm:main}, we obtain the following corollary.
\begin{corollary}\label{coro:tt-cpu-cost}
Consider the setting of Theorem~\ref{thm:main}. If $\sqrt{\alpha} - \sqrt{\beta}>\sqrt{2}$, then for all sufficiently large $n$, the probability that Algorithm~\ref{alg:TSA} outputs $\bxs$ or $-\bxs$ in $\mathcal{O}(n\log^2 n/\log\log n)$ time is at least $1 - n^{-\Omega(1)}$. 
\end{corollary}

To put the above results in perspective, let us make the following remarks:
\begin{itemize}
\item[(a)] While Problem \eqref{eq:MLE} is known to be NP-hard in the worst case, the assumption that the adjacency matrix $\bA$ arises from the binary symmetric SBM in Definition~\ref{model:SBM} allows us to conduct an average-case analysis of Algorithm \ref{alg:TSA}. In particular, if the constants $\alpha,\beta>0$ satisfy $\sqrt{\alpha} - \sqrt{\beta} > \sqrt{2}$, which is known to be the information-theoretic limit for exact recovery~\cite{abbe2016exact,mossel2014consistency}, then with high probability, $\bxs$ and $-\bxs$ are the only optimal solutions of Problem \eqref{eq:MLE}~\cite{abbe2016exact}. Moreover, Corollary~\ref{coro:tt-cpu-cost} shows that with high probability, Algorithm~\ref{alg:TSA} computes an optimal solution of Problem~\eqref{eq:MLE} in nearly-linear time. As such, Algorithm \ref{alg:TSA} is more efficient than SDP-based methods (see, e.g., \cite{bandeira2018random,hajek2016achieving}). The time bound for Algorithm \ref{alg:TSA} is also competitive with those for some of the most efficient methods in the literature (see, e.g., \cite{abbe2015community,gao2017achieving,mossel2014consistency,yun2016optimal}) under the setting of Theorem \ref{thm:main}.

\item[(b)] In the recent work~\cite{abbe2017entrywise}, Abbe et al. showed that the vanilla spectral method, which first computes an exact eigenvector $\bm{u}_2$ associated with the second-largest eigenvalue of $\bA$ and then returns $\widehat{\bm x}={\rm sgn}(\bm{u}_2)$ as the label vector, achieves exact recovery under the setting of Theorem \ref{thm:main}. Conceptually, the method can be implemented in the framework of Algorithm~\ref{alg:TSA} as follows. First, by performing $N\rightarrow\infty$ orthogonal iterations, we obtain a limit point $\bar{\bQ}^{(\infty)}$, which can be used to construct $\bm{u}_2$; see \cite[Theorem 8.2.2]{golub2012matrix} and compare with lines~\ref{line:evec-1}--\ref{line:evec-2} of Algorithm~\ref{alg:TSA}. Then, we return the label $\widehat{\bm x}={\rm sgn}(\bm{u}_2)$. Incidentally, observe that if the label $\widehat{\bm x}$ coincides (up to sign) with the ground-truth label $\bxs$, then $\bm{u}_2$ has exactly $n/2$ positive entries and $n/2$ negative entries. This, together with Proposition \ref{prop:cf-proj}, implies that $\widehat{\bm x}$ can also be computed by projecting $\bu_2$ onto $\mF$.

In actual implementation, however, we need to know when to terminate the orthogonal iterations, so that the vanilla spectral method can proceed to the sign-taking step. Unfortunately, the results in~\cite{abbe2017entrywise} do not provide the required termination criterion. By contrast, Theorem \ref{alg:TSA} shows that the underlying communities can be exactly recovered by first performing $\mO(\log n/\log\log n)$ orthogonal iterations to obtain a coarse approximation $\bm{y}$ of $\bm{u}_2$ and then applying $\mO(\log n/\log\log n)$ projected power iterations to a suitably scaled $\bm{y}$. As it turns out, by combining our results in Sections~\ref{sec:oim} and~\ref{sec:ppm} with the arguments in~\cite{abbe2017entrywise}, we can show that the vanilla spectral method only needs $\mO(\log n/\log\log n)$ orthogonal iterations before the sign-taking step in order to exactly recover the underlying communities with high probability; see Section~\ref{sec:spectral-time}. It is worth noting that this gives the best complexity bound known to date for the vanilla spectral method in the context of exact community recovery in the binary symmetric SBM. This further demonstrates the power of our approach.
%
%
\end{itemize}

\section{Proof of the Main Result} \label{sec:pf-main}
In this section, we prove our main result (i.e., Theorem~\ref{thm:main}) concerning the recovery performance and iteration complexity of Algorithm \ref{alg:TSA}. This involves establishing some key properties of the orthogonal iterations and projected power iterations, which will be  accomplished in Sections~\ref{sec:oim} and~\ref{sec:ppm}, respectively.

\subsection{Analysis of the Method of Orthogonal Iteration}\label{sec:oim}
Our main goal in this sub-section is to provide a probabilistic analysis of the convergence behavior of the orthogonal iterations deployed in the first stage of Algorithm~\ref{alg:TSA}. Such an analysis is not only useful for proving Theorem \ref{thm:main} but may also be of independent interest. To proceed, let us introduce some further notation that will be used in the sequel. Let $\bA$ be as in Theorem~\ref{thm:main} and consider its eigenvalue decomposition $\bA = \bU\bLbd\bU^T$, where $\bLbd = \diag(\lambda_1,\dots,\lambda_n)$ with $\lambda_1\geq \dots\geq \lambda_n$ and $\bU = [\bu_1,\dots,\bu_n]$. We write
\begin{equation}\label{eq:a-b}
\begin{aligned}
\bLbd_\alpha & := \diag(\lambda_1,\lambda_2), & \bLbd_\beta & := \diag(\lambda_3,\dots,\lambda_n), \\
\bU_\alpha & := [\bu_1,\bu_2], & \bU_\beta & := [\bu_3,\dots,\bu_n].
\end{aligned}
\end{equation} 
Moreover, we define,
for every $k\geq 0$,  
\begin{equation}
\label{eq:PVW}
\bP^{(k)} := \bU^T\bQ^{(k)}, \quad \bV^{(k)} := \bU_{\alpha}^T\bQ^{(k)}, \quad \bW^{(k)} := \bU_{\beta}^T\bQ^{(k)},
\end{equation} 
where the sequence $\{\bQ^{(k)}\}_{k\geq 1}$ is generated by Algorithm \ref{alg:TSA}. Note that for all $k\geq0$, $\bP^{(k)}\in\R^{n\times 2}$, $\bV^{(k)}\in\R^{2\times 2}$, $\bW^{(k)}\in\R^{(n-2)\times 2}$, and $\bP^{(k)}$ has orthonormal columns. Besides, by the CS decomposition (see, e.g.,~\cite[Theorem 2.5.2]{golub2012matrix}), we have
\begin{equation}
\label{eq:min+max}
\sigma_{\min}^2(\bV^{(k)}) + \sigma_{\max}^2(\bW^{(k)}) = 1,\quad k=0,1,\dots.
\end{equation}

It is known that the quantity $\sqrt{1-\sigma_{\min}^2(\bV^{(0)})}$ measures the distance between the subspaces spanned by the columns of $\bU_\alpha$ and $\bQ^{(0)}$; see, e.g., \cite[Theorem 2.5.1]{golub2012matrix}. To bound this distance, we prove the following result.
\begin{lemma}\label{lem:ortho-init}
For all $n \ge 6$, it holds with probability at least $1 - 4\sqrt{\log n/n}$ that
\begin{equation}
\label{dist:ortho-init}
\sigma_{\min}(\bV^{(0)}) \ge \frac{1}{\sqrt{n^2+1}}.
\end{equation}
\end{lemma}
\begin{proof}
Recall from Algorithm \ref{alg:TSA} that $\bQ^{(0)} = \bY(\bY^T\bY)^{-1/2}$, where $\bY \in \R^{n\times 2}$ is a random matrix whose entries are i.i.d.~standard normal random variables. By the definition of $\bU_\alpha$, we have $\bU_\alpha = \bU\bE$, where $\bE := (\bee_1,\bee_2)$. This, together with $\bV^{(0)} = \bU_{\alpha}^T\bQ^{(0)}$, $\bQ^{(0)} = \bY(\bY^T\bY)^{-1/2}$, $\bU\bU^T = \bI$, and the orthogonal invariance of the normal distribution, yields
\begin{align*}
\bV^{(0)} = \bE^T(\bU^T\bY)\left( (\bU^T\bY)^T(\bU^T\bY) \right)^{-1/2} \overset{d}{=} \bE^T\bY ( \bY^T\bY)^{-1/2} = \bE^T\bQ^{(0)}.
\end{align*}
It then follows that
\begin{align}
& \P\left( \sigma_{\min}\left(\bV^{(0)}\right) \le a \right) = \P\left( \sigma_{\min}\left(\bE^T\bQ^{(0)}\right) \leq a \right) \nonumber \\ & \qquad = \P\left( \lambda_{\min}\left({\bQ^{(0)}}^T\bE\bE^T\bQ^{(0)}\right) \leq a^2 \right)  
= \int_0^{a^2} p(x)dx, \label{eq:integral}
\end{align}  
where $p(\cdot)$ is the probability density function of $\lambda_{\min}\left({\bQ^{(0)}}^T\bE\bE^T\bQ^{(0)}\right)$. By~\cite[eq.~(6)]{absil2006largest}, $p(\cdot)$ takes the form
\begin{equation}
\label{eq:pdf}
p(x) = \frac{(n-2)(n-3)}{4}\cdot\frac{(1-x)^{n-3}}{x}\int_0^1y\left(1 + \frac{1-x}{x}y\right)^{-\frac{1}{2}}(1-y)^{\frac{n-5}{2}}dy
\end{equation}
for any $x\in(0,1]$.
This allows us to derive an upper bound on the integral in \eqref{eq:integral}. Indeed, using the Cauchy-Schwarz inequality, we get
\begin{align*}
& \int_0^1y\left(1 + \frac{1-x}{x}y\right)^{-\frac{1}{2}}(1-y)^{\frac{n-5}{2}}dy \\ & \quad \leq \left(\int_0^1\left(1 + \frac{1-x}{x}y\right)^{-1}dy\right)^{\frac{1}{2}}\cdot \left(\int_0^1 y^2(1-y)^{n-5}dy\right)^{\frac{1}{2}} \\
& \quad = \left( \frac{x}{1-x}\cdot\log\left(\frac{1}{x}\right)\right)^\frac{1}{2}\cdot \left(\frac{2}{(n-2)(n-3)(n-4)}\right)^{\frac{1}{2}}.
\end{align*} 
Together with \eqref{eq:pdf}, this implies that for all $n\geq 6$, 
\begin{align}
\int_0^{a^2} p(x)dx & \leq \frac{\sqrt{2}}{4}\cdot\frac{\sqrt{(n-2)(n-3)}}{\sqrt{n-4}}\cdot\int_0^{a^2}\sqrt{\frac{1}{x}\log\left(\frac{1}{x}\right)}(1-x)^{n-\frac{7}{2}}dx \nonumber \\
& \leq \frac{\sqrt{2n}}{4}\cdot\int_0^{a^2}\sqrt{\frac{1}{x}\log\left(\frac{1}{x}\right)}dx, \label{eq:int--1}
\end{align}
where the second inequality follows from the fact that for all $n\geq 6$, we have $(n-2)(n-3)\leq n(n-4)$ and $0\leq (1-x)^{n-7/2}\leq 1$ for every $0\leq x\leq 1$. By letting $x = e^{-2z^2}$, we obtain
\begin{align}
\int_0^{a^2}\sqrt{\frac{1}{x}\log\left(\frac{1}{x}\right)}dx = 4\sqrt{2}\int_{\sqrt{\log\left(\frac{1}{a}\right)}}^\infty z^2e^{-z^2}dz. \label{eq:int-0}
\end{align} 
Moreover, using integration by parts, we compute
\begin{equation}\label{eq:int-1}
\int_{\sqrt{\log\left(\frac{1}{a}\right)}}^\infty z^2e^{-z^2}dz = \frac{a}{2}\sqrt{\log\left(\frac{1}{a}\right)} + \frac{1}{2}\int_{\sqrt{\log\left(\frac{1}{a}\right)}}^\infty e^{-z^2}dz.
\end{equation}
Notice that for any $a \leq 1/e$, we have $\sqrt{\log(1/a)} \geq 1$, which leads to 
\begin{equation}\label{eq:int-2}
\int_{\sqrt{\log\left(\frac{1}{a}\right)}}^\infty e^{-z^2}dz \leq \int_{\sqrt{\log\left(\frac{1}{a}\right)}}^\infty z^2 e^{-z^2}dz. 
\end{equation}
Combining \eqref{eq:int-1} and \eqref{eq:int-2}, we have that for any $a\leq 1/e$, 
$$ \int_{\sqrt{\log\left(\frac{1}{a}\right)}}^\infty z^2e^{-z^2}dz \leq a\sqrt{\log\left(\frac{1}{a}\right)}. $$
This, together with \eqref{eq:integral}, \eqref{eq:int--1}, and \eqref{eq:int-0}, yields
$$ \P\left( \sigma_{\min}\left(\bV^{(0)}\right) \le a \right) \leq 2\sqrt{n}a\sqrt{\log\left(\frac{1}{a}\right)} $$
for all $n\geq 6$ and $a\leq 1/e$. By letting $a = 1/\sqrt{n^2+1}$ (which satisfies $a\leq 1/e$ for any $n\geq 6$) in the above inequality, we obtain 
$$ \P\left( \sigma_{\min}\left(\bV^{(0)}\right) \leq \frac{1}{\sqrt{n^2+1}} \right) \leq \frac{2\sqrt{n}}{\sqrt{n^2+1}}\sqrt{\log\left(\sqrt{n^2+1}\right)} \le 4\sqrt{\frac{\log n}{n}}, $$
which implies Lemma \ref{lem:ortho-init} as desired.
\end{proof}

\smallskip
Next, we present a spectral bound on the deviation of $\bA$ from its mean. It is a direct consequence of \cite[Theorem 5.2]{lei2015consistency} and thus we omit its proof. 
\begin{lemma}\label{lem:A-2norm}
There exist constants $c_1\ge 1$ and $c_2>0$, whose values depend only on $\alpha$ and $\beta$, such that
\begin{equation}
\label{eq:A-2norm}
\|\bA - \E[\bA]\| \leq c_1\sqrt{\log n}
\end{equation}
holds with probability at least $1 - c_2n^{-3}$.
\end{lemma}
Based on Lemma \ref{lem:A-2norm}, we can establish the following corollary, which provides estimates on the eigenvalues and eigenvectors of $\bA$.
\begin{corollary}\label{coro:eig-bound}
With probability at least $1 - c_2n^{-3}$, the following statements hold:
\begin{align}
\begin{split}\label{eq:lambda1}
& \frac{\alpha+\beta}{2}\log n - c_1\sqrt{\log n} \le \lambda_1 \le \frac{\alpha+\beta}{2}\log n + c_1\sqrt{\log n},
\end{split}\\
\begin{split}\label{eq:lambda2}
& \frac{\alpha-\beta}{2}\log n - c_1\sqrt{\log n}\le \lambda_2 \le \frac{\alpha-\beta}{2}\log n + c_1\sqrt{\log n},
\end{split} \\
\begin{split}\label{eq:lambdai}
& |\lambda_i| \le c_1\sqrt{\log n}, \quad i=3,\dots,n,
\end{split} \\
\begin{split}\label{eq:u1-pert}
& \min_{ \theta \in \{\pm 1\} } \left \| \theta\bu_1 - \frac{\bo}{\sqrt{n}} \right\|_2 \le \frac{c_3}{\sqrt{\log n}},
\end{split} \\
\begin{split}\label{eq:u2-pert}
& \min_{ \theta \in \{\pm 1\} } \left \| \theta\bu_2 - \frac{\bxs}{\sqrt{n}} \right\|_2 \le \frac{c_3}{\sqrt{\log n}},
\end{split}
\end{align}
where $c_1$, $c_2$ are the constants in Lemma \ref{lem:A-2norm} and $c_3 := 2\sqrt{2}c_1/\min\{\beta,(\alpha-\beta)/2\}$.
\end{corollary}
\begin{proof}
Suppose that the statement in Lemma \ref{lem:A-2norm} holds, which happens with probability at least $1 - c_2n^{-3}$. Let $\nu_1\geq \nu_2 \geq \dots \geq \nu_n$ be the eigenvalues of $\E[\bA]$. It follows from Weyl's inequality (see, e.g.,~\cite[Theorem 4.5.3]{vershynin2018high}) that
\begin{equation}
\label{eq:weyl}
|\lambda_i - \nu_i| \leq \|\bA - \E[\bA]\|, \quad i=1,2,\dots,n.
\end{equation}
According to the binary symmetric SBM in Definition~\ref{model:SBM}, we have
\begin{equation}
\label{eq:lam-exp}
\E[\bA] = \frac{p+q}{2}\bo\bo^T + \frac{p-q}{2}{\bx^*}{\bx^*}^T.
\end{equation}
Since $\bo^T\bxs = 0$ and $\|\bo\|_2^2=\|\bxs\|_2^2=n$, we see that $\E[\bA]$ is a rank-2 matrix with $\bo$ and $\bxs$ being the eigenvectors associated with the largest and second-largest eigenvalues, respectively. Using \eqref{eq:p-q}, we can compute
\begin{equation}\label{eq:nu-eig-EA}
\nu_1 = \frac{\alpha+\beta}{2}\log n, \quad \nu_2 = \frac{\alpha-\beta}{2}\log n, \quad \nu_i = 0, \quad i=3,\dots,n.
\end{equation}
By \eqref{eq:weyl}, \eqref{eq:nu-eig-EA}, and Lemma \ref{lem:A-2norm}, the desired results \eqref{eq:lambda1}--\eqref{eq:lambdai} are immediate.
%
Moreover, it follows from \eqref{eq:nu-eig-EA} that
$$ \delta_1:=\min_{i\neq 1}|\nu_1 - \nu_i| = \beta\log n, \quad \delta_2:=\min_{i\neq 2}|\nu_2 - \nu_i| =  \min\left\{\beta, \frac{\alpha-\beta}{2}\right\}\log n.$$
This, together with the Davis-Kahan theorem (see, e.g., \cite[Theorem 4.5.5]{vershynin2018high}) and Lemma \ref{lem:A-2norm}, yields 
$$ \min_{\theta\in\{\pm 1\}} \left\|\theta\bu_1 - \frac{\bo}{\sqrt{n}}\right\|_2 \leq \frac{2\sqrt{2}\|\bA - \mathbb{E}[\bA]\|}{\delta_1} \leq \frac{2\sqrt{2}c_1}{\beta\sqrt{\log n}}\leq \frac{c_3}{\sqrt{\log n}}$$
and
$$ \min_{\theta\in\{\pm 1\}} \left\|\theta\bu_2 - \frac{\bxs}{\sqrt{n}}\right\|_2 \leq \frac{2\sqrt{2}\|\bA - \mathbb{E}[\bA]\|}{\delta_2} \leq \frac{2\sqrt{2}c_1}{\min\left\{ \beta, \frac{\alpha-\beta}{2} \right\} \sqrt{\log n}} = \frac{c_3}{\sqrt{\log n}}. $$
The proof is then completed.
\end{proof}

\smallskip
Now, we are ready to analyze the convergence of the orthogonal iterations. From Algorithm~\ref{alg:TSA}, one can verify by induction that for every $k\geq 1$, 
\begin{equation}
\label{eq:OI-k}
\bA^k\bQ^{(0)} = \bQ^{(k)}\tbR^{(k)},
\end{equation}
where $\tbR^{(k)}:=\bR^{(k)}\bR^{(k-1)}\cdots \bR^{(1)}$ with $\tbR^{(1)} = \bR^{(1)}$. 
This, together with \eqref{eq:a-b}--\eqref{eq:PVW} and $\bA = \bU\bLbd\bU^T$, yields
\begin{align}
\bLbd_\alpha^k\bV^{(0)} = \bV^{(k)}\tbR^{(k)}, \quad 
\bLbd_\beta^k\bW^{(0)} = \bW^{(k)}\tbR^{(k)}. \label{eq:V-W}
\end{align}
Suppose that \eqref{dist:ortho-init} and $\lambda_1\geq\lambda_2>0$ hold. Then, $\bLbd_\alpha^k\bV^{(0)}$ is non-singular, which implies that the square matrix $\bV^{(k)}$ is invertible and $\tbR^{(k)} = (\bV^{(k)})^{-1}\bLbd_\alpha^k\bV^{(0)}$. Together with \eqref{eq:V-W}, this leads to
\begin{equation}
\label{eq:K-rate}
\bK^{(k)} = \bLbd^k_\beta\bK^{(0)}\bLbd^{-k}_\alpha,
\end{equation}
where we define $\bK^{(k)}:= \bW^{(k)}(\bV^{(k)})^{-1}$ for all $k\geq 0$. Let $\bk^{(k)}_1$ and $\bk^{(k)}_2$ be the first and second column of $\bK^{(k)}$, respectively. The following result characterizes the convergence rates of $\bk^{(k)}_1$ and $\bk^{(k)}_2$.
\begin{proposition}\label{prop:K-rate}
Suppose that $n\geq\exp(16c_1^2/(\alpha-\beta)^2)$ and that \eqref{dist:ortho-init}, \eqref{eq:lambda1}--\eqref{eq:lambdai}, and $\lambda_2>0$ hold. Then, for every $k\geq 0$, it holds that
\begin{align}\label{eq:k-rate-n}
\|\bk_1^{(k)}\|_2 \leq n\left(\frac{4c_1}{(\alpha+\beta)\sqrt{\log n}}\right)^k \quad \mbox{and} \quad 
\|\bk_2^{(k)}\|_2 \leq n\left(\frac{4c_1}{(\alpha-\beta)\sqrt{\log n}}\right)^k. 
\end{align}
\end{proposition}
\begin{proof}
It follows from \eqref{eq:K-rate} and $\lambda_1\geq \lambda_2 >0$ that
\begin{equation}
\label{eq:k-rate}
\|\bk_1^{(k)}\|_2 \leq \left(\frac{\bar{\lambda}}{\lambda_1}\right)^k\|\bK^{(0)}\|, \quad \|\bk_2^{(k)}\|_2 \leq \left(\frac{\bar{\lambda}}{\lambda_2}\right)^k\|\bK^{(0)}\|,
\end{equation}
where $\bar{\lambda} = \max\{|\lambda_3|,\dots,|\lambda_n|\}$. By \eqref{eq:lambdai}, we have $\bar{\lambda} \leq c_1\sqrt{\log n}$. Moreover, 
$$ \|\bK^{(0)}\| = \|\bW^{(0)}(\bV^{(0)})^{-1}\| 
\leq  \frac{\sigma_{\max}(\bW^{(0)})}{\sigma_{\min}(\bV^{(0)})} = \sqrt{\frac{1}{\sigma^2_{\min}(\bV^{(0)})}-1} \leq n, $$
where the second equality follows from \eqref{eq:min+max} and the last inequality is due to \eqref{dist:ortho-init}. These, together with \eqref{eq:lambda1} and \eqref{eq:lambda2}, yield
\begin{align}
\|\bk_1^{(k)}\|_2 & \leq n\left(\frac{2c_1\sqrt{\log n}}{(\alpha+\beta)\log n - 2c_1\sqrt{\log n}}\right)^k, \nonumber \\
\|\bk_2^{(k)}\|_2 & \leq n\left(\frac{2c_1\sqrt{\log n}}{(\alpha-\beta)\log n - 2c_1\sqrt{\log n}}\right)^k. \nonumber
\end{align}
In particular, the desired result \eqref{eq:k-rate-n} holds for all $n$ satisfying $4c_1\sqrt{\log n}\leq (\alpha-\beta)\log n$, or equivalently, $n\geq\exp(16c_1^2/(\alpha-\beta)^2)$.
\end{proof}

\smallskip
Next, we study the effect of Ritz acceleration. Let $\{\bar{\bQ}^{(k)}\}_{k\ge1}^{N-1}$ be an auxiliary sequence constructed via $\bar{\bQ}^{(k)}=\bQ^{(k)}\bm{H}^{(k)}$, where $\bm{H}^{(k)}$ is obtained from the eigen-decomposition of $\bQ^{(k)^T} \bA {\bQ}^{(k)}$---i.e., $\bQ^{(k)^T} \bA {\bQ}^{(k)} = \bH^{(k)}\bD^{(k)}{\bH^{(k)}}^T$ with $\bD^{(k)} = \mbox{diag}(d_1^{(k)},d_2^{(k)})$ and $|d_1^{(k)}| \geq |d_2^{(k)}|$. We are interested in relating the Ritz eigenvalues $d_1^{(k)}, d_2^{(k)}$ to the eigenvalues $\lambda_1,\lambda_2$ of $\bA$. Towards that end, we define 
\[ \overline{d}^{(k)} := \max\{ d_1^{(k)},d_2^{(k)} \}, \quad \underline{d}^{(k)} := \min\{ d_1^{(k)},d_2^{(k)} \}. \]
Then, we have the following estimates.
\begin{proposition}\label{prop:rate-EV}
For every $k\geq 0$, it holds that
\begin{equation}
\label{eq:rate-EV}
\lambda_1 \geq \overline{d}^{(k)} \geq \lambda_1 - 2\|\bA\| \cdot \|\bk_1^{(k)}\|_2^2, \quad \lambda_2 \geq \underline{d}^{(k)}.
\end{equation}
\end{proposition}
\begin{proof}
For ease of exposition, we shall omit the superscript $(k)$ in $d_1^{(k)}$, $d_2^{(k)}$, $\overline{d}^{(k)}$, $\underline{d}^{(k)}$, $\bQ^{(k)}$, $\bP^{(k)}$, $\bV^{(k)}$, $\bW^{(k)}$, $\bk_1^{(k)}$, and $\bK^{(k)}$ throughout the proof.
Since $\overline{d} \geq \underline{d}$ are the eigenvalues of $\bQ^T\bA\bQ$ and $\lambda_1\geq \lambda_2$ are the largest two eigenvalues of $\bA$, we have $\lambda_1\geq \overline{d}$ and $\lambda_2\geq \underline{d}$ from the Courant-Fischer minimax theorem (see, e.g., \cite[Theorem 8.1.2]{golub2012matrix}). Moreover, using \eqref{eq:a-b}--\eqref{eq:PVW} and the definition of $\bK$ and letting $\he_1 = (1,0)$, we have 
\[ (\bV^{-1}\he_1)^T\bQ^T\bA\bQ(\bV^{-1}\he_1) = \he_1^T(\bLbd_\alpha + \bK^T\bLbd_\beta\bK)\he_1 = \lambda_1 + \bk_1^T\bLbd_\beta\bk_1 \]
and
\begin{equation}
\label{eq:V-1-e}
\|\bV^{-1}\he_1\|_2^2 = (\bV^{-1}\he_1)^T\bP^T\bP\bV^{-1}\he_1 = \he_1^T(\bI + \bK^T\bK)\he_1 = 1 + \|\bk_1\|_2^2, 
\end{equation} 
where the first equality in~\eqref{eq:V-1-e} is due to the fact that $\bP$ has orthonormal columns. Now, since $\overline{d}$ is the largest eigenvalue of $\bQ^T\bA\bQ$, we have
$$ \overline{d} \geq \frac{(\bV^{-1}\he_1)^T\bQ^T\bA\bQ(\bV^{-1}\he_1)}{\|\bV^{-1}\he_1\|_2^2} = \frac{\lambda_1 + \bk_1^T\bLbd_\beta\bk_1}{1 + \|\bk_1\|_2^2},  $$
which leads to
$$ \overline{d} - \lambda_1 \geq \frac{-\lambda_1\|\bk_1\|_2^2 + \bk_1^T\bLbd_\beta\bk_1}{1 + \|\bk_1\|_2^2} \geq -\frac{2\|\bA\| \cdot \|\bk_1\|_2^2}{1 + \|\bk_1\|_2^2}\geq -2\|\bA\| \cdot \|\bk_1\|_2^2. $$
The proof is then completed.
\end{proof}

\smallskip
%
Now, let $\overline{\bh}^{(k)}, \underline{\bh}^{(k)}$ denote the eigenvectors associated with the eigenvalues $\overline{d}^{(k)}$, $\underline{d}^{(k)}$, respectively. Furthermore, we define $\overline{\bq}^{(k)} := \bQ^{(k)}\overline{\bh}^{(k)}$ and $\underline{\bq}^{(k)} := \bQ^{(k)}\underline{\bh}^{(k)}$. Equipped with Propositions \ref{prop:K-rate} and \ref{prop:rate-EV}, we can establish the convergence rate of the orthogonal iterations.
\begin{theorem}
\label{thm:OI}
Let $c_1\ge1$ and $c_2>0$ be the constants in Lemma \ref{lem:A-2norm}. Suppose that 
\begin{equation}
\label{eq:n-bd-1}
n \ge \max\left\{\exp\left(\frac{16c_1^2}{\beta^2}\right), 6\right\}. 
\end{equation} 
Then, it holds with probability at least $1 - c_2n^{-3} - 4\sqrt{\log n/n}$ that
\begin{align}
\min_{\theta\in\{\pm 1\}}\|\overline{\bq}^{(k)} - \theta\bu_1\|_2 & \leq n\left(\frac{8\alpha}{\beta}+10\right)\left(\frac{4c_1}{(\alpha+\beta)\sqrt{\log n}}\right)^k,\quad \forall\ k \geq 0, \label{eq:q1-rate} \\
\min_{\theta\in\{\pm 1\}}\|\underline{\bq}^{(k)} - \theta\bu_2\|_2 & \leq n\left(\frac{16\alpha}{\beta}+18\right)\left(\frac{4c_1}{(\alpha-\beta)\sqrt{\log n}}\right)^k,\quad \forall\ k \geq \bar{K}, \label{eq:q2-rate} 
\end{align}
where
\begin{equation}
\label{eq:bK}
\bar{K} = \left\lceil \frac{2\log n + \log\left(\frac{8(\alpha+\beta)}{\beta}\right)}{\log\log n + 2\log\left(\frac{\alpha+\beta}{4c_1}\right)}    \right\rceil.
\end{equation}
\end{theorem}
\begin{proof}
Again, we shall omit the superscript $(k)$ throughout the proof for ease of exposition. By definition, we have
\begin{equation}
\label{eq:eig-H}
\bQ^T\bA\bQ\overline{\bh} = \overline{d}\,\overline{\bh}, \quad \bQ^T\bA\bQ\underline{\bh} = \underline{d}\,\underline{\bh}.
\end{equation}
Observe that for any $\bv\in\R^2$ with $\|\bv\|_2=1$, we have
\begin{align*}
& \|\overline{\bh} - \mbox{sgn}(\overline{\bh}^T\bv)\bv\|_2^2 = 2 - 2\mbox{sgn}(\overline{\bh}^T\bv) \overline{\bh}^T\bv = 2 - 2|\overline{\bh}^T\bv| \nonumber \\
& \quad \le 2 - 2(\overline{\bh}^T\bv)^2 = 2(\underline{\bh}^T\bv)^2, 
\end{align*}
where the first equality follows from $\|\overline{\bh}\|_2=1$, the second equality uses $|a| = \mbox{sgn}(a)\cdot a$ for any $a\in\R$, the inequality follows from $|\overline{\bh}^T\bv| \le \|\overline{\bh}\|_2\|\bv\|_2 = 1$, and the last equality is due to $\|\bv\|_2=1$ and $\bH$ being orthogonal.
This gives
\begin{align} \label{eq:claim}
\|\overline{\bh} - \mbox{sgn}(\overline{\bh}^T\bv)\bv\|_2 \le \sqrt{2}|\underline{\bh}^T\bv|
\end{align}
for any $\bv\in\R^2$ with $\|\bv\|_2=1$. 
Moreover, we obtain from \eqref{eq:eig-H} that
\begin{align}
& |(\lambda_1-\underline{d})\underline{\bh}^T\bV^{-1}\he_1| = |\lambda_1\underline{\bh}^T\bV^{-1}\he_1 - \underline{\bh}^T\bQ^T\bA\bQ\bV^{-1}\he_1| \nonumber \\
& \quad \leq \|\lambda_1\bV^{-1}\he_1 - \bP^T\bLbd\bP\bV^{-1}\he_1\|_2 \leq \|\lambda_1\bP\bV^{-1}\he_1 - \bLbd\bP\bV^{-1}\he_1\|_2 \nonumber \\
& \quad = \left\| \lambda_1 \begin{bmatrix}
\bI \\ \bK
\end{bmatrix}\he_1 - \begin{bmatrix}
\bLbd_\alpha & \bm{0} \\
\bm{0} & \bLbd_\beta
\end{bmatrix}\begin{bmatrix}
\bI \\ \bK
\end{bmatrix}\he_1\right\|_2 = \|\lambda_1\bk_1 - \bLbd_\beta\bk_1\|_2 \nonumber \\
& \quad\leq 2\|\bA\| \cdot \|\bk_1\|_2, \label{eq:lam-d}
\end{align} 
where the first inequality is due to $\|\underline{\bh}\|_2=1$ and $\bQ^T\bA\bQ = \bP^T\bLbd\bP$, the second inequality follows from $\bP^T\bP=\bI$ and $\|\bP^T\bv\|_2\le \|\bv\|_2$, and the second equality uses \eqref{eq:PVW} and $\bK = \bW\bV^{-1}$. 
Then, by letting $\bar{\theta} = \mbox{sgn}(\overline{\bh}^T\bV^{-1}\he_1)$, we have
\begin{equation}
\label{eq:step-1a}
\left\|\overline{\bh} - \bar{\theta}\frac{\bV^{-1}\he_1}{\|\bV^{-1}\he_1\|_2} \right\|_2 \leq \frac{{\sqrt{2}} |\underline{\bh}^T\bV^{-1}\he_1|}{\|\bV^{-1}\he_1\|_2} \leq \frac{{2\sqrt{2}}\|\bA\| \cdot \|\bk_1\|_2}{|\lambda_1 - \underline{d}|},
\end{equation}
where the first inequality follows from \eqref{eq:claim} and the second one is due to \eqref{eq:lam-d} and $\|\bV^{-1}\he_1\|_2 \geq 1$ (implied by \eqref{eq:V-1-e}). Also, by \eqref{eq:PVW} and $\bU^T\bu_1 = \bee_1$, we obtain
\begin{equation}
\label{eq:step-2a}
\|\bQ\bV^{-1}\he_1 - \bu_1\|_2 = \|\bP\bV^{-1}\he_1 - \bee_1\|_2 = \left\|\begin{bmatrix}
\bI \\ \bK
\end{bmatrix}\he_1 - \bee_1\right\|_2 = \|\bk_1\|_2.
\end{equation}
It then follows that
\begin{align}
& \min_{\theta\in\{\pm 1\}}\|\overline{\bq} - \theta\bu_1\|_2 \leq \|\overline{\bq} - \bar{\theta}\bu_1\|_2 = \|\bQ\overline{\bh} - \bar{\theta}\bu_1\|_2 \nonumber \\
& \ \leq \left\|\bQ\overline{\bh} - \bar{\theta}\frac{\bQ\bV^{-1}\he_1}{\|\bV^{-1}\he_1\|_2}\right\|_2 
+\left\|\bar{\theta}\frac{\bQ\bV^{-1}\he_1}{\|\bV^{-1}\he_1\|_2} - \bar{\theta}\frac{\bu_1}{\|\bV^{-1}\he_1\|_2}\right\|_2 \nonumber \\
& \quad+ \left\|\bar{\theta}\frac{\bu_1}{\|\bV^{-1}\he_1\|_2} - \bar{\theta}\bu_1\right\|_2 \nonumber \\
& \ \leq \left\|\overline{\bh} - \bar{\theta}\frac{\bV^{-1}\he_1}{\|\bV^{-1}\he_1\|_2} \right\|_2 + \frac{\|\bQ\bV^{-1}\he_1 - \bu_1\|_2}{\|\bV^{-1}\he_1\|_2} + \left|\frac{1}{\|\bV^{-1}\he_1\|_2} - 1\right| \nonumber \\
& \ \leq \frac{2\sqrt{2}\|\bA\| \cdot \|\bk_1\|_2}{|\lambda_1 - \underline{d}|} + \frac{\|\bk_1\|_2}{\sqrt{1 + \|\bk_1\|_2^2}} + \frac{\sqrt{1 + \|\bk_1\|_2^2}-1}{\sqrt{1 + \|\bk_1\|_2^2}} \nonumber \\
& \ \leq \left(\frac{{2\sqrt{2}}\|\bA\|}{|\lambda_1 - \underline{d}|} + 2\right)\|\bk_1\|_2, \label{eq:q1-u1}
\end{align}
where the third inequality uses the fact that $\bQ$ has orthonormal columns, $|\bar{\theta}|=1$, and $\|\bu_1\|_2=1$; the fourth one follows from \eqref{eq:V-1-e}, \eqref{eq:step-1a}, and \eqref{eq:step-2a}; the last one uses $\sqrt{1 + \|\bk_1\|_2^2}\geq 1$ and $\sqrt{1 + \|\bk_1\|_2^2} -1\leq \|\bk_1\|_2$.

By a similar argument, one can show that
\begin{equation}
\label{eq:q2-u2}
\min_{\theta\in\{\pm 1\}}\|\underline{\bq} - \theta\bu_2\|_2 \leq \left(\frac{{2\sqrt{2}}\|\bA\|}{|\lambda_2 - \overline{d}|} + 2 \right)\|\bk_2\|_2.
\end{equation}
Now, suppose that \eqref{dist:ortho-init}, \eqref{eq:lambda1}--\eqref{eq:lambdai}, and $\lambda_2>0$ hold, which happens with probability at least $1 - c_2n^{-3} - 4\sqrt{\log n/n}$ for all $n$ satisfying \eqref{eq:n-bd-1} due to Lemma \ref{lem:ortho-init}, Corollary \ref{coro:eig-bound}, and the union bound. Then, by \eqref{eq:lambda1}, \eqref{eq:lambda2}, Proposition \ref{prop:rate-EV}, and $n\geq \exp({16}c_1^2/\beta^2)$, we obtain 
\begin{equation}
\label{eq:bd-l1-l2}
|\lambda_1 - \underline{d}| \geq \lambda_1 - \lambda_2 \geq \beta\log n - {2}c_1\sqrt{\log n} \geq \frac{\beta}{2}\log n.
\end{equation} 
Moreover, it follows from \eqref{eq:lambda1}--\eqref{eq:lambdai} and $n\geq \exp(4c_1^2/(\alpha+\beta)^2)$ that
\begin{equation}
\label{eq:bd-A}
\|\bA\| = \lambda_1 \leq \frac{\alpha+\beta}{2}\log n + c_1\sqrt{\log n} \leq (\alpha+\beta)\log n. 
\end{equation} 
These, together with the first inequality in \eqref{eq:k-rate-n}, imply \eqref{eq:q1-rate} as desired. Besides, using $ { n > \exp(16c_1^2/(\alpha+\beta)^2) }$ and \eqref{eq:k-rate-n}, one can verify that for every $k\geq \bar{K}$, 
$$ \|\bk_1^{(k)}\|_2
\leq n\left(\frac{4c_1}{(\alpha+\beta)\sqrt{\log n}}\right)^{\bar{K}} \leq \sqrt{\frac{\beta}{8(\alpha+\beta)}} \leq \sqrt{\frac{\lambda_1 - \lambda_2}{4\|\bA\|}}, $$
where the last inequality follows from \eqref{eq:bd-l1-l2} and \eqref{eq:bd-A}. It then follows from Proposition \ref{prop:rate-EV} that $\lambda_1 - \overline{d}^{(k)} \leq 2\|\bA\| \cdot \|\bk_1\|_2^2 \leq (\lambda_1 - \lambda_2)/2$ for all $k\geq \bar{K}$. Together with \eqref{eq:bd-l1-l2}, this implies that for all $k\geq \bar{K}$,
\begin{equation}
\label{eq:bd-d1-l2}
\overline{d}^{(k)} - \lambda_2 = \overline{d}^{(k)} - \lambda_1 + \lambda_1 - \lambda_2 \geq \frac{\lambda_1 - \lambda_2}{2} \geq \frac{\beta}{4}\log n. 
\end{equation} 
The desired result {\eqref{eq:q2-rate}} then follows from \eqref{eq:k-rate-n}, {\eqref{eq:q2-u2}}, \eqref{eq:bd-A}, and \eqref{eq:bd-d1-l2}.
\end{proof}

\subsection{Analysis of the Method of Projected Power Iteration}\label{sec:ppm}
Now, let us turn to study the convergence behavior of the projected power iterations employed in the second stage of Algorithm \ref{alg:TSA}. Our goal is to show that if $\bx^{(0)}$ is properly chosen, then the projected power iterations will terminate in a finite number of iterations and output the ground-truth label vector $\bxs$ with high probability. To begin, let $\bA$ be as in Theorem~\ref{thm:main}. In particular, Lemma \ref{lem:A-2norm} and Corollary \ref{coro:eig-bound} can be applied here. 

Recall that the projected power iterations take the form 
\[ 
\bx^{(k)} \in \mP(\bA\bx^{(k-1)}), \quad k=1,2,\dots,
\] 
where $\mathcal{P}$ is the projection operator onto $\mathcal{F}$; see~\eqref{eq:def-FS}--\eqref{eq:PPI}. The following result shows that $\mathcal{P}$ possesses a Lipschitz-like property, despite the fact that $\mathcal{F}$ is a discrete set. Such a property plays an important role in the analysis of the projected power iterations; cf.~\cite{liu2017estimation,liu2017discrete}.

\begin{lemma}
\label{lem:Proj-cont}
Suppose that $\bc\in\mathbb{R}^n$ is arbitrary and $\varepsilon > 0$ is constant such that 
\begin{equation}
\label{eq:v2}
c_i \ \left\{ \begin{aligned}
&\geq \varepsilon, \quad \ \ i\in\mcI, \\
&\leq -\varepsilon, \quad i\in[n]\setminus\mcI
\end{aligned}\right.	
\end{equation}
for some $\mcI\subset[n]$ with $|\mcI| = n/2$.
Then, for any $\bv\in\mP(\bc)$, $\bcp\in\mathbb{R}^n$, and $\bvp\in\mP(\bcp)$, it holds that
\begin{equation}
\label{eq:Proj-cont}
 \|\bv - \bvp\|_2 \leq \frac{2\|\bc - \bcp\|_2}{\varepsilon}.
\end{equation}
\end{lemma}
\begin{proof}
By \eqref{eq:v2} and Proposition \ref{prop:cf-proj}, we see that $\mP(\bc)$ is a singleton and $\bv\in\mP(\bc)$ satisfies
\begin{equation}
\label{eq:vvv}
v_i = \left\{
\begin{array}{rl}
1, & \quad i\in\mcI, \\
-1, & \quad i\in[n]\setminus\mcI.
\end{array}\right.
\end{equation}
Let $\bcp\in\R^n$ be arbitrary and $\bvp \in \mP(\bcp)$. It then follows from Proposition \ref{prop:cf-proj} that 
\begin{equation}
\label{eq:vvvp}
v^\prime_i = \left\{
\begin{array}{rl}
1, & \quad i\in\mcJ, \\
-1, & \quad i\in[n]\setminus\mcJ
\end{array}\right.
\end{equation}
for some $\mcJ\subset[n]$ with $|\mcJ| = n/2$ such that $c^\prime_i\geq c^\prime_j$ for all $i\in\mcJ$ and $j\in[n]\setminus\mcJ$. For ease of exposition, we write $\mcI^c:=[n]\setminus\mcI$ and $\mcJ^c:=[n]\setminus\mcJ$. Since 
\[ |\mcI \cap \mcJ| + |\mcI \cap \mcJ^c| = |\mcI| = \frac{n}{2}, \quad |\mcI \cap \mcJ| + |\mcI^c \cap \mcJ| = |\mcJ| = \frac{n}{2}, \]
we deduce that $|\mcI\cap\mcJ^c| = |\mcI^c\cap\mcJ| = s$ for some $0\leq s\leq n/2$. In addition, by \eqref{eq:vvv} and \eqref{eq:vvvp}, we have
\begin{equation*}
v_i - v^\prime_i = \left\{\begin{array}{rl}
0, & \quad i\in (\mcI\cap\mcJ) \cup (\mcI^c\cap\mcJ^c), \\
2, & \quad i\in\mcI\cap\mcJ^c, \\
-2, & \quad i\in\mcI^c\cap\mcJ.
\end{array}\right.
\end{equation*}
Since $|\mcI\cap\mcJ^c| = |\mcI^c\cap\mcJ| = s$, this yields
\begin{equation}
\label{eq:v-diff}
\|\bv - \bvp\|_2 = 2\sqrt{2s}.
\end{equation}
On the other hand, we have
\begin{equation}\label{eq:c-cp-1}
\|\bc - \bcp\|_2^2 = \sum_{i=1}^n(c_i - c^\prime_i)^2  \geq \sum_{i\in\mcI\cap\mcJ^c}(c_i - c^\prime_i)^2 + \sum_{j\in\mcI^c\cap\mcJ}(c_j - c^\prime_j)^2.
\end{equation}
It follows from \eqref{eq:v2} that $c_i - c_j \geq 2\varepsilon$ for any $i\in\mcI$ and $j\in\mcI^c$. Besides, recall from \eqref{eq:vvvp} that $c^\prime_j \geq c^\prime_i$ for any $i\in\mcJ^c$ and $j\in\mcJ$.
Thus, for every $i\in\mcI\cap\mcJ^c$ and $j\in\mcI^c\cap\mcJ$, it holds that
$$ (c_i - c^\prime_i)^2 + (c_j - c^\prime_j)^2 \geq \frac{1}{2}(\underbrace{c_i - c_j}_{\geq 2\varepsilon} + \underbrace{c^\prime_j - c^\prime_i}_{\geq 0})^2 \geq 2\varepsilon^2, $$
where the first inequality is due to $a^2 + b^2 \geq (a-b)^2/2$ for any $a,b\in\R$. Using \eqref{eq:c-cp-1} and $|\mcI\cap\mcJ^c| = |\mcI^c\cap\mcJ| = s$, we obtain
\begin{equation}
\label{eq:c-diff}
\|\bc - \bcp\|_2^2 \geq 2s\varepsilon^2.
\end{equation}
The desired result \eqref{eq:Proj-cont} then follows from \eqref{eq:v-diff} and \eqref{eq:c-diff}.
\end{proof}

\smallskip
Next, we recall the following result, which is established in \cite{abbe2016exact} and pertains to the difference of two binomial random variables; see also \cite[Lemma 8]{abbe2017entrywise}.

\begin{lemma}
\label{lem:binom-tail}
Let $\alpha>\beta>0$ be given constants. Suppose that $\{W_i\}_{i=1}^{n/2}$ are i.i.d.~$\mathbf{Bern}(\alpha\log n/n)$ and $\{Z_i\}_{i=1}^{n/2}$ are i.i.d.~$\mathbf{Bern}(\beta\log n/n)$, with $\{Z_i\}_{i=1}^{n/2}$ being independent of $\{W_i\}_{i=1}^{n/2}$. Then, for any $\gamma\in\R$, it holds that
\begin{equation}
\label{eq:binom-tail}
\P\left( \sum_{i=1}^{n/2}W_i - \sum_{i=1}^{n/2}Z_i \leq \gamma\log n\right) \leq n^{-\frac{(\sqrt{\alpha} - \sqrt{\beta})^2}{2} + \frac{\gamma(\log\alpha-\log\beta)}{2}}.
\end{equation}
\end{lemma}

Equipped with Lemmas \ref{lem:Proj-cont} and \ref{lem:binom-tail}, we can show that the set-valued map $\bx \Mapsto \mP(\bA\bx)$ possesses a contraction property in a certain neighborhood of $\bxs$. This would then imply the linear convergence of the projected power iterations.

\begin{proposition}\label{prop:contrac-rate}
Suppose that the constants $\alpha>\beta>0$ satisfy $\sqrt{\alpha} - \sqrt{\beta} > \sqrt{2}$. 
Then, there exists a constant $\gamma > 0$, whose value depends only on $\alpha$ and $\beta$, such that the following statement holds with probability at least $1-n^{-\Omega(1)}$: For all $\bx \in \R^n$ such that $\bo^T\bx = 0$, $\|\bx\|_2=\sqrt{n}$, and 
\be\label{cond:lem-contrac-rate}
\|\bx-\bx^*\|_2 \le 5c_3\sqrt{\frac{n}{\log n}},
\ee
one has
\be\label{result:lem-contrac-rate}
\left\| \bv - \bx^* \right\|_2 \le \frac{c_4}{\gamma\sqrt{\log n}} \|\bx-\bx^*\|_2
\ee
for any $\bv\in\mP(\bA\bx)$, where $c_4:=3c_3(\alpha-\beta) + 2c_1$ and $c_1$, $c_3$ are the constants in Lemma \ref{lem:A-2norm} and Corollary \ref{coro:eig-bound}, respectively.
\end{proposition}
\begin{proof}
Since $\sqrt{\alpha} - \sqrt{\beta} > \sqrt{2}$, there exists a $\gamma > 0$, whose value depends only on $\alpha$ and $\beta$, such that 
\begin{equation}
\label{eq:c5}
c_5:=\frac{(\sqrt{\alpha} - \sqrt{\beta})^2}{2} - \frac{\gamma(\log\alpha - \log\beta)}{2}-1>0.
\end{equation}
According to the binary symmetric SBM in Definition \ref{model:SBM}, we have
$$ x_i^*(\bA\bxs)_i = \sum_{j=1}^n a_{ij}x_i^*x_j^* \overset{d}{=} \sum_{i=1}^{n/2}W_i - \sum_{i=1}^{n/2}Z_i $$
for every $i\in[n]$, where $\{W_i\}_{i=1}^{n/2}$ are i.i.d.~$\mathbf{Bern}(\alpha\log n/n)$, $\{Z_i\}_{i=1}^{n/2}$ are i.i.d. $\mathbf{Bern}(\beta\log n/n)$, and $\{Z_i\}_{i=1}^{n/2}$ are independent of $\{W_i\}_{i=1}^{n/2}$.
It then follows from Lemma \ref{lem:binom-tail}, \eqref{eq:c5}, and the union bound that
\begin{equation}
\label{eq:opt-tail-bd-B}	
\min_{i \in [n]} x_i^*(\bA\bxs)_i \geq \gamma\log n
\end{equation}
holds with probability at least $1 - n^{-c_5}$. 

In the rest of the proof, we suppose that both \eqref{eq:A-2norm} and \eqref{eq:opt-tail-bd-B} hold, which happens with probability at least 
{$1 - c_2n^{-3} - n^{-c_5}$} due to Lemma \ref{lem:A-2norm} and the union bound. Let $\mcI=\{i\in[n]: x_i^*=1\}$ and $\mcI^c = [n]\setminus\mcI$. Since $\bxs\in\mcF$, we have $|\mcI| = |\mcI^c| = n/2$ and $x^*_i = -1$ for all $i\in\mcI^c$. This, together with \eqref{eq:opt-tail-bd-B}, implies that
\begin{equation}
\label{eq:Ax}
(\bA\bxs)_i \left\{\begin{aligned}
&\geq\gamma\log n, \quad \ \ \, i\in\mcI, \\
&\leq-\gamma\log n, \quad i\in\mcI^c.
\end{aligned}\right.
\end{equation} 
It then follows from Proposition \ref{prop:cf-proj} that $\mP(\bA\bx^*) = \{\bx^*\}$. Now, let $\bx \in \R^n$ be such that $\bo^T\bx = 0$, $\|\bx\|_2=\sqrt{n}$, and \eqref{cond:lem-contrac-rate} holds. By \eqref{eq:Ax}, $\mP(\bA\bx^*) = \{\bx^*\}$, and Lemma~\ref{lem:Proj-cont}, we obtain
\begin{equation}
\label{eq:v-xs}
\|\bv - \bxs\|_2 \leq \frac{2\|\bA\bx - \bA\bxs\|_2}{\gamma\log n}
\end{equation} 
for any $\bv\in\mP(\bA\bx)$. In addition, using $\|\bxs\|_2 = \|\bx\|_2$, we compute
\begin{align*}
\|\bx - \bxs\|_2^2 = \|\bx\|_2^2 + \|\bxs\|_2^2 - 2\bx^T\bxs = 2\|\bxs\|_2^2 - 2\bx^T\bxs = 2( {\bxs}^T\bxs - \bx^T\bxs ).
\end{align*}
This, together with \eqref{eq:A-2norm}, \eqref{eq:lam-exp}, $\bo^T\bx=\bo^T\bxs=0$, and $\|\bxs\|_2=\sqrt{n}$, yields 
\begin{align}
	\|\bA\bx - \bA\bxs\|_2 & = \|\mathbb{E}[\bA](\bx - \bxs) + (\bA - \mathbb{E}[\bA])(\bx - \bxs)\|_2 \nonumber \\
	& \leq \left\| \left(\frac{p+q}{2}\bo\bo^T + \frac{p-q}{2}{\bx^*}{\bx^*}^T\right)(\bx-\bxs)\right\|_2 \nonumber \\
	& \quad+ \left\|\bA - \mathbb{E}[\bA]\right\|\cdot\left\|\bx-\bxs \right\|_2 \nonumber \\
	& = \left\| \frac{p-q}{2}\left( \bx^T\bxs -{\bxs}^T\bxs\right)\bxs\right\|_2 + \left\|\bA - \mathbb{E}[\bA]\right\| \cdot \left\|\bx-\bxs \right\|_2 \nonumber \\
	& \le \frac{p-q}{4}\sqrt{n} \|\bx - \bxs\|_2^2 + c_1\sqrt{\log n} \| \bx-\bxs \|_2. \label{eq:Ax-Axs}
\end{align}
Then, by \eqref{eq:p-q}, \eqref{cond:lem-contrac-rate}, \eqref{eq:v-xs}, and \eqref{eq:Ax-Axs}, we obtain for any $\bv\in\mP(\bA\bx)$ that 
\begin{align*}
\left\| \bv - \bxs \right\|_2
& \le \frac{2\|\bA\bx-\bA\bxs\|_2}{\gamma\log n} \\
& \leq \frac{2}{\gamma\log n}\left(\frac{(\alpha-\beta)\log n}{4\sqrt{n}}\|\bx - \bxs\|_2 + c_1\sqrt{\log n}\right)\|\bx - \bxs\|_2 \\
& \le \frac{ 3c_3(\alpha-\beta) + 2c_1 }{\gamma\sqrt{\log n}}\|\bx-\bxs\|_2 \\
& = \frac{c_4}{\gamma\sqrt{\log n}}\|\bx - \bxs\|_2.
\end{align*}
This completes the proof.
\end{proof}

\smallskip
Since the feasible set $\mF$ of Problem~\eqref{eq:MLE} is discrete, the contraction property~\eqref{result:lem-contrac-rate} suggests that if an iterate is sufficiently close to $\bxs$, then the projected power iterations will exhibit one-step convergence to $\bxs$; i.e., all subsequent iterates will stay at $\bxs$. This would then imply the finite termination of stage 2 of Algorithm~\ref{alg:TSA}. Let us now formalize the above observation.

\begin{proposition}\label{prop:one-step}
Suppose that the constants $\alpha>\beta>0$ satisfy $\sqrt{\alpha} - \sqrt{\beta} > \sqrt{2}$ and let $\gamma > 0$ be such that \eqref{eq:c5} holds. 
Then, the following statement holds with probability at least $1-n^{-\Omega(1)}$: For all $\bx \in \mcF$ such that $\|\bx-\bxs\|_2 \le \sqrt{2 \gamma\log n }$, one has
\be\label{result:lem-one-step}
\mP(\bA\bx) = \{\bxs\}.
\ee
\end{proposition}
\begin{proof}
Suppose that \eqref{eq:opt-tail-bd-B} holds, which happens with probability at least $1 - n^{-c_5}$, where $c_5>0$ is given in \eqref{eq:c5}. Let $\mcI=\{i\in[n]: x_i^*=1\}$ and $\mcI^c = [n]\setminus\mcI$. Since $\bxs\in\mcF$, we have $|\mcI| = |\mcI^c| = n/2$ and $x^*_i = -1$ for all $i\in\mcI^c$. Let $\bx\in\mcF$ be such that $\|\bx - \bxs\|_2 \leq \sqrt{2\gamma\log n}$. Since $\bx,\bxs\in\mcF$, there exist an $S\subset\mcI$ and an $S^\prime\subset\mcI^c$ with $|S| = |S^\prime| \leq n/2$ such that
\[
\bx = \bxs - 2\bm{e}_{S} + 2\bm{e}_{S^\prime},
\]
where $\bm{e}_{S}$ (resp.~$\bm{e}_{S^\prime}$) is an $n$-dimensional vector with $(\bm{e}_{S})_i=1$ if $i\in S$ (resp.~$S^\prime$) and $0$ otherwise. Observe that
\[ \sqrt{2\gamma\log n} \ge \|\bx - \bxs\|_2 = 2\| \bm{e}_{S^\prime} - \bm{e}_{S} \|_2 = 2\sqrt{|S|+|S^\prime|}. \]
This, together with $|S|=|S^\prime|$, implies that $|S| = |S^\prime| \le \gamma\log n/4$. Now, for all $i\in \mcI$, we get
\begin{align*}
(\bA\bx)_i & = (\bA\bxs)_i - 2(\bA\bm{e}_S)_i + 2(\bA\bm{e}_{S^\prime})_i \\
& \ge \gamma \log n - 2\sum_{j\in S} a_{ij} + 2 \sum_{j\in S^\prime} a_{ij}  \\
& \ge \gamma \log n - 2|S|  \\
& \ge \frac{1}{2}\gamma \log n,
\end{align*}
where the first inequality follows from \eqref{eq:opt-tail-bd-B}, the second inequality uses $0\leq a_{ij} \le 1$, and the last one is due to $|S| \le \gamma\log n/4$. By the same argument, we can show that for all $i\in\mcI^c$,
\[
(\bA\bx)_i \le -\frac{1}{2}\gamma \log n.
\]
These, together with Proposition \ref{prop:cf-proj}, imply \eqref{result:lem-one-step} as desired.
\end{proof}

\smallskip
The next result establishes the iteration complexity of the projected power iterations for finding the ground-truth label $\bxs$. Recall that $c_1$, $c_3$, and $c_4$ are the constants defined in Lemma \ref{lem:A-2norm}, Corollary \ref{coro:eig-bound}, and Proposition \ref{prop:contrac-rate}, respectively, whose values depend only on $\alpha$ and $\beta$.

\begin{theorem}\label{thm:PPI-bd}
Suppose that the constants $\alpha>\beta>0$ satisfy $\sqrt{\alpha} - \sqrt{\beta} > \sqrt{2}$ and let $\gamma > 0$ be such that \eqref{eq:c5} holds. Suppose in addition that
\begin{equation}
\label{eq:n-bd-3}
n > \max\left\{ \exp\left(\frac{c_4^2}{\gamma^2}\right), \exp\left(\frac{5c_3}{\sqrt{2\gamma}}\right) \right\}.
\end{equation}
Let $\{\bx^{(k)}\}_{k\geq 0}$ be the sequence generated by Algorithm \ref{alg:TSA}. If $\bx^{(0)}$ satisfies 
\begin{equation}
\label{eq:x0-cond}
\bo^T\bx^{(0)} =0, \quad \|\bx^{(0)}\|_2 =\sqrt{n}, \quad \|\bx^{(0)} - \bxs \|_2 \le 5c_3\sqrt{\frac{n}{\log n}},
\end{equation}
then the following statements hold with probability at least $1-n^{-\Omega(1)}$:
\begin{itemize}
\item[(i)] For all $k\geq 1$, it holds that
\be\label{eq:contrac-rate}
\| \bx^{(k)} - \bx^* \|_2 \le \frac{c_4}{\gamma\sqrt{\log n}} \|\bx^{(k-1)}-\bx^*\|_2 
\ee
and
\begin{equation}
\label{eq:xk-range}
\|\bx^{(k)} - \bxs \|_2 \le 5c_3\sqrt{\frac{n}{\log n}}.
\end{equation}
\item[(ii)] There exists some $k\leq \tilde{K}$ such that $\bx^{(k)} = \bx^{(k-1)} = \bxs$, where
\begin{equation}
\label{eq:PPI-bd}
{\tilde{K}= \left\lceil  \frac{\log n}{\log\log n + 2\log\left( \frac{\gamma}{c_4} \right)} \right\rceil + 2}.
\end{equation}
\end{itemize}
\end{theorem}
\begin{proof}
Suppose that the statements in Propositions \ref{prop:contrac-rate} and \ref{prop:one-step} hold simultaneously, which happens with probability at least $1 - n^{-\Omega(1)}$ by the union bound. We first prove (i). It follows from \eqref{eq:x0-cond} that \eqref{eq:xk-range} holds for $k=0$. Moreover, by \eqref{eq:x0-cond}, Proposition \ref{prop:contrac-rate}, and $\bx^{(1)}\in\mP(\bA\bx^{(0)})$, one can observe that \eqref{eq:contrac-rate} holds for $k=1$. These, together with \eqref{eq:n-bd-3}, yield
$$ \| \bx^{(1)} - \bx^* \|_2 \le \frac{c_4}{\gamma\sqrt{\log n}} \|\bx^{(0)}-\bx^*\|_2 \leq 5c_3\sqrt{\frac{n}{\log n}} $$
and thus \eqref{eq:xk-range} holds for $k=1$. Then, (i) can be established by a simple inductive argument. Next, we prove (ii). By \eqref{eq:n-bd-3}, we have $c_4/(\gamma\sqrt{\log n}) < 1$ and $5c_3/\sqrt{\log n} < \sqrt{2\gamma\log n}$. This, together with \eqref{eq:x0-cond}, \eqref{eq:contrac-rate}, and \eqref{eq:PPI-bd}, yields 
$$ 
\begin{aligned}
& \|\bx^{(\tilde{K}-2)} - \bxs\|_2 \leq \|\bx^{(0)} - \bxs\|_2 \left(\frac{c_4}{\gamma\sqrt{\log n}}\right)^{\tilde{K}-2} \\
& \quad \leq 5c_3\sqrt{\frac{n}{\log n}}\left(\frac{c_4}{\gamma\sqrt{\log n}}\right)^{\frac{\log n}{\log\log n + 2\log(\gamma/c_4)}} = \frac{5c_3}{\sqrt{\log n}} \leq \sqrt{2\gamma\log n}.
\end{aligned}
$$ 
By Proposition \ref{prop:one-step} and the projected power iterations~\eqref{eq:PPI}, we have $\bx^{(\tilde{K}-1)} = \bxs$. By applying Proposition \ref{prop:one-step} to $\bx^{(\tilde{K}-1)}$, we further have $\bx^{(\tilde{K})} = \bx^{(\tilde{K}-1)} = \bxs$. This, together with the stopping criterion in Algorithm \ref{alg:TSA}, completes the proof of (ii).
\end{proof}

\smallskip
Let us make two remarks before we leave this sub-section. First, due to the symmetry of Problem \eqref{eq:MLE}, Theorem \ref{thm:PPI-bd} also holds if we replace all the $\bxs$ in its statement by $-\bxs$. Second, Theorem \ref{thm:PPI-bd}(ii) implies that the second stage of Algorithm \ref{alg:TSA} terminates in a finite number of iterations at an optimal solution of Problem \eqref{eq:MLE}, provided that $\bx^{(0)}$ is properly chosen. In particular, for any $\bx^{(0)}$ satisfying \eqref{eq:x0-cond}, it terminates in roughly $\mathcal{O}(\log n/\log\log n)$ iterations for all sufficiently large $n$.

\subsection{Proofs of Theorem \ref{thm:main} and Corollary \ref{coro:tt-cpu-cost}}

With the preparations in Sections~\ref{sec:oim} and~\ref{sec:ppm}, we are ready to establish the main results stated in Section \ref{sec:main-results}. We first provide a formal version of Theorem \ref{thm:main} and its proof. Recall that the constants $c_1, c_3, c_4$ are given in Lemma \ref{lem:A-2norm}, Corollary \ref{coro:eig-bound}, and Proposition \ref{prop:contrac-rate}, respectively, and their values depend only on $\alpha$ and $\beta$. 

\begin{theorem}\label{thm-1}
Consider the setting of Theorem~\ref{thm:main}. Suppose that $\sqrt{\alpha} - \sqrt{\beta} > \sqrt{2}$ and let $\gamma>0$ be such that \eqref{eq:c5} holds. In addition, suppose that
\begin{equation}
\label{eq:n-bd-2}
n > \max\left\{\exp\left(\frac{16c_1^2}{\min\{(\alpha-\beta)^2,\beta^2\}}\right), \exp\left( \frac{5c_3}{\sqrt{2\gamma}} \right), \exp\left(\frac{c_4^2}{\gamma^2}\right), 6 \right\}. 
\end{equation}
Then, with probability at least $1 - n^{-\Omega(1)}$, Algorithm \ref{alg:TSA} takes at most $N_1$ orthogonal iterations and $N_2$ projected power iterations to find $\bx^*$ or $-\bxs$, where
\begin{equation}
\label{eq:N1}
N_1 = \left\lceil  \frac{2\log n + \log\log n + 2\log\left(\frac{16\alpha+18\beta}{\beta\cdot\min\{1, c_3\}}\right)}{\log\log n + 2\log\left(\frac{\alpha-\beta}{4c_1}\right)} \right\rceil,
\end{equation}
\begin{equation}
\label{eq:N2}
N_2 = \left\lceil  \frac{\log n}{\log\log n + 2\log\left(\frac{\gamma}{c_4}\right)} \right\rceil + 2.
\end{equation}
\end{theorem}
\begin{proof}
Suppose that the statements in Lemma \ref{lem:A-2norm}, Theorem \ref{thm:OI}, and Theorem \ref{thm:PPI-bd} hold, which happens with probability at least $1-n^{-\Omega(1)}$ due to the union bound. Recall from line~\ref{line:oi-end} of Algorithm \ref{alg:TSA} that $\bar{\bQ}^{(N_1) = \bQ^{(N_1)}\bH^{(N_1)}}$. Moreover, the columns of $\bar{\bQ}^{(N_1)}$ are denoted by $\overline{\bq}^{(N_1) = \bQ^{(N_1)}\overline{\bh}^{(N_1)}}$ and $ \underline{\bq}^{(N_1) = \bQ^{(N_1)}\underline{\bh}^{(N_1)}}$, where $\overline{\bh}^{(N_1)},\underline{\bh}^{(N_1)}$ are the eigenvectors of $ \bQ^{{(N_1)}^T \bA {\bQ}^{(N_1)}}$ associated with the eigenvalues $\overline{d}^{(N_1)}, \underline{d}^{(N_1)}$, respectively, and $\overline{d}^{(N_1)} \ge \underline{d}^{(N_1)}$.

By $\alpha>\beta>0$, \eqref{eq:bK}, and \eqref{eq:N1}, one can verify that $N_1 \ge \bar{K}$. It then follows from Theorem \ref{thm:OI} and \eqref{eq:N1} that there exist $\theta_1,\theta_2\in\{\pm 1\}$ satisfying
\begin{align}
\|\overline{\bq}^{(N_1)} - \theta_1\bu_1\|_2 & \leq n\left(\frac{8\alpha}{\beta}+10\right)\left(\frac{4c_1}{(\alpha+\beta)\sqrt{\log n}}\right)^{N_1} \leq \frac{c_3}{\sqrt{\log n}}, \label{eq:q1-rate-1} \\
\|\underline{\bq}^{(N_1)} - \theta_2\bu_2\|_2 & \leq n\left(\frac{16\alpha}{\beta}+18\right)\left(\frac{4c_1}{(\alpha-\beta)\sqrt{\log n}}\right)^{N_1} \leq \frac{c_3}{\sqrt{\log n}}. \label{eq:q2-rate-1}
\end{align}
After obtaining $\bar{\bQ}^{(N_1)}$, Algorithm \ref{alg:TSA} generates an initial point $\bx^{(0)}$ for the projected power iterations by setting $\bx^{(0)} = \sqrt{n}\by^{(0)}/\|\by^{(0)}\|_2$, where $\by^{(0)} = \underline{\bq}^{(N_1)} - (\bo^T\underline{\bq}^{(N_1)}/n)\bo$. Since $\|\bo\|_2 = \sqrt{n}$ and $\|\underline{\bq}^{(N_1)}\|_2 = 1$, we see that $\|\by^{(0)}\|_2 \leq 1$ and
$$ {\by^{(0)}}^T\underline{\bq}^{(N_1)} = 1 - \frac{(\bo^T\underline{\bq}^{(N_1)})^2}{n} \geq 0. $$
These, together with $\|\bx^{(0)}\|_2 = \sqrt{n}$, yield
$$
\begin{aligned}
\|\bx^{(0)} -  \sqrt{n}\underline{\bq}^{(N_1)}\|_2^2 & = 2n - 2n\frac{(\by^{(0)})^T \underline{\bq}^{(N_1)}}{\|\by^{(0)}\|_2} \leq 2n - 2n(\by^{(0)})^T\underline{\bq}^{(N_1)} \\
& = 2n - 2n\left(1 - \frac{(\bo^T\underline{\bq}^{(N_1)})^2}{n}\right) = 2(\bo^T\underline{\bq}^{(N_1)})^2,
\end{aligned}
$$
which implies that
\begin{align}\label{eq-1:thm-1}
\|\bx^{(0)} -  \sqrt{n}\underline{\bq}^{(N_1)} \|_2 \le \sqrt{2} | \bo^T\underline{\bq}^{(N_1)} |.
\end{align}
Besides, by \eqref{eq:u1-pert}, there exist $\tilde{\theta}_1, \tilde{\theta}_2 \in\{\pm 1\}$ such that 
\begin{equation}
\label{eq:u1u2-pert}
\left\|\tilde{\theta}_1\bu_1 - \frac{\bo}{\sqrt{n}}\right\|_2 \leq \frac{c_3}{\sqrt{\log n}}, \ \quad \left\|\tilde{\theta}_2\bu_2 - \frac{\bxs}{\sqrt{n}}\right\|_2 \leq \frac{c_3}{\sqrt{\log n}}.
\end{equation} 
This, together with \eqref{eq:q1-rate-1}, \eqref{eq-1:thm-1}, $(\overline{\bq}^{(N_1)})^T\underline{\bq}^{(N_1)} = 0$, and $\|\underline{\bq}^{(N_1)}\|_2=1$, yields
\begin{align}
\|\bx^{(0)} - \sqrt{n}\underline{\bq}^{(N_1)}\|_2 & \leq \sqrt{2n} \left| \left(\frac{\bo }{\sqrt{n}}- \tilde{\theta}_1\bu_1 + \tilde{\theta}_1\bu_1 - \frac{\tilde{\theta}_1}{\theta_1}\overline{\bq}^{(N_1)} \right)^T \underline{\bq}^{(N_1)} \right| \nonumber \\
& \leq \sqrt{2n}\left( \left\|\tilde{\theta}_1\bu_1 - \frac{\bo}{\sqrt{n}}\right\|_2 + \left\|\theta_1\bu_1 - \overline{\bq}^{(N_1)}\right\|_2 \right) \nonumber \\
& \leq 2\sqrt{2}c_3\sqrt{\frac{n}{\log n}} \label{eq:x0-q2}.
\end{align}
Then, by \eqref{eq:q2-rate-1}, \eqref{eq:u1u2-pert}, \eqref{eq:x0-q2}, and $\theta_2,\tilde{\theta}_2\in\{\pm 1\}$, we obtain
\begin{align*}
& \min_{\theta\in\{\pm 1\}} \|\bx^{(0)} - \theta\bxs\|_2 \leq \left\|\bx^{(0)} - \frac{\theta_2}{\tilde{\theta}_2}\bxs \right\|_2 \\
& \quad \le \| \bx^{(0)} - \sqrt{n}\underline{\bq}^{(N_1)}\|_2 + \sqrt{n}\|\underline{\bq}^{(N_1)} - \theta_2\bu_2\|_2 + \sqrt{n}\left\|\theta_2\bu_2 - \frac{\theta_2}{\tilde{\theta}_2}\frac{\bx^*}{\sqrt{n}}\right\|_2 \\
& \quad = \| \bx^{(0)} - \sqrt{n}\underline{\bq}^{(N_1)}\|_2 + \sqrt{n}\|\underline{\bq}^{(N_1)} - \theta_2\bu_2\|_2 + \sqrt{n}\left\|\tilde{\theta}_2\bu_2 - \frac{\bx^*}{\sqrt{n}}\right\|_2 \\
& \quad \le (2+2\sqrt{2})c_3\sqrt{\frac{n}{\log n}} \leq 5c_3\sqrt{\frac{n}{\log n}}.
\end{align*}
This, together with $\|\bx^{(0)}\|_2=\sqrt{n}$, $\bo^T\bx^{(0)}=0$, and Theorem \ref{thm:PPI-bd}, implies that if $N_2=\tilde{K}$, where $\tilde{K}$ is defined in \eqref{eq:PPI-bd}, then Algorithm \ref{alg:TSA} can find $\bxs$ or $-\bxs$ with probability at least $1-n^{-\Omega(1)}$.
\end{proof}

\smallskip
Armed with the results in Theorem~\ref{thm-1}, we can now provide a proof of Corollary \ref{coro:tt-cpu-cost}.

\begin{proof}[of Corollary \ref{coro:tt-cpu-cost}]
From the discussion following Theorem~\ref{thm:main}, it remains to bound the number of non-zero entries in $\bA$.  Towards that end, let us first estimate the number of non-zero entries in an arbitrary column $\bm{a} \in \mathbb{R}^n$ of $\bA$, which we denote by $\|\bm{a}\|_0$. According to the binary symmetric SBM in Definition \ref{model:SBM}, we have 
\[
\|\bm{a}\|_0 \overset{d}{=} \sum_{i=1}^{n/2} W_i + \sum_{i=1}^{n/2} Z_i,
\]
where $\{W_i\}_{i=1}^{n/2}$ are i.i.d.~$\mathbf{Bern}(p)$, $\{Z_i\}_{i=1}^{n/2}$ are i.i.d.~$\mathbf{Bern}(q)$, and $\{Z_i\}_{i=1}^{n/2}$ are independent of $\{W_i\}_{i=1}^{n/2}$. It follows that 
$$\E[\|\bm{a}\|_0] = \frac{n}{2}(p+q), \qquad \mathrm{Var}(\|\bm{a}\|_0) \leq \frac{n}{2}(p+q). $$ 
Upon applying Bernstein's inequality for bounded distributions (see, e.g.,~\cite[Theorem 2.8.4]{vershynin2018high}), we get
\begin{align*}
& \P\left( \left| \|\bm{a}\|_0-\frac{n}{2}(p+q) \right| \ge \frac{3}{2}n(p+q) \right) \le 2\exp\left( - \frac{\frac{9}{8}n^2(p+q)^2}{\frac{1}{2}n(p+q)+\frac{1}{2}n(p+q)} \right) \\
& \quad = 2\exp\left( -\frac{9}{8}n(p+q) \right) = 2\exp\left( -\frac{9}{8}(\alpha+\beta)\log n \right) = 2n^{-\frac{9}{8}(\alpha+\beta)}.
\end{align*}
This gives
\[
\P\left( \|\bm{a}\|_0 < 2n(p+q) \right) \ge 1 - 2n^{-\frac{9}{8}(\alpha+\beta)} \geq 1 - 2n^{-\frac{9}{4}},
\]
where the second inequality is due to $\alpha + \beta >2$, which follows from $\alpha>\beta>0$ and $\sqrt{\alpha} - \sqrt{\beta}>\sqrt{2}$. 

Now, upon applying the union bound, we conclude that with probability at least $1 - 2n^{-5/4}$, the number of non-zero entries in $\bA$ is less than $2n^2(p+q) = 2(\alpha+\beta)n\log n$. Thus, the per-iteration cost of Algorithm \ref{alg:TSA} is $\mathcal{O}(n\log n)$. Together with Theorem \ref{thm:main}, the desired bound on the total computational cost of Algorithm \ref{alg:TSA} follows.
\end{proof}

\subsection{Discussion} \label{subsec:discuss}
Before we leave this section, let us discuss two further consequences of our technical development. 

\subsubsection{Binary Symmetric SBM without Self-Loops} \label{sec:no-loop}
Consider a variant of the binary symmetric SBM in which no self-loop is allowed in the graph. Specifically, the entries of the adjacency matrix $\bA$ are still generated independently according to~\eqref{eq:SBM}, except that $a_{ii}=0$ for all $i\in[n]$. Our results are still valid under this setting. Indeed, it is easy to verify that 
\[ \E[\bA] = \frac{p+q}{2}\bo\bo^T + \frac{p-q}{2}{\bx^*}{\bx^*}^T - p\bm{I}. \]
In particular, the eigenvalues of $\E[\bA]$ change by $p=\mO(\log n/n)$, which is dominated by other quantities that appear Lemma~\ref{lem:A-2norm} and Corollary~\ref{coro:eig-bound}. Thus, the results in Section~\ref{sec:oim} still hold. Moreover, observe that
\[ x_i^*(\bA\bxs)_i = \sum_{j\not=i} a_{ij}x_i^*x_j^* \overset{d}{=} \sum_{i=1}^{n/2-1}W_i - \sum_{i=1}^{n/2}Z_i \]
for every $i\in[n]$, where, as before, $\{W_i\}_{i=1}^{n/2-1}$ are i.i.d.~$\mathbf{Bern}(\alpha\log n/n)$, $\{Z_i\}_{i=1}^{n/2}$ are i.i.d. $\mathbf{Bern}(\beta\log n/n)$, and $\{Z_i\}_{i=1}^{n/2}$ are independent of $\{W_i\}_{i=1}^{n/2-1}$. Since
\[ \P\left( \sum_{i=1}^{n/2-1}W_i - \sum_{i=1}^{n/2}Z_i \leq \gamma\log n\right) \le \P\left( \sum_{i=1}^{n/2-1}W_i - \sum_{i=1}^{n/2-1}Z_i \leq \gamma\log n + 1 \right) \]
for any $\gamma\in\mathbb{R}$, we can apply Lemma~\ref{lem:binom-tail} and obtain the results in Section~\ref{sec:ppm}. Hence, we can conclude that Theorem~\ref{thm-1} holds. Lastly, the number of non-zero entries in $\bA$ can only decrease if there is no self-loop, which implies that Corollary~\ref{coro:tt-cpu-cost} also holds.

\subsubsection{Implementation of the Vanilla Spectral Method and Its Complexity Analysis} \label{sec:spectral-time}

Recently, Abbe et al.~\cite{abbe2017entrywise} have shown that the vanilla spectral method, which is presented in Algorithm~\ref{alg:spectral}, can also achieve exact recovery down to the information-theoretic limit in the binary symmetric SBM. A popular and practically efficient way of implementing this method is to employ orthogonal iterations to find the eigenvector $\bm{u}_2$ of the adjacency matrix $\bA$. However, the results in~\cite{abbe2017entrywise} do not establish the number of orthogonal iterations needed to obtain a sufficiently accurate approximation of $\bm{u}_2$ that can provably recover the communities in the graph. 

\begin{algorithm}[htb]
\DontPrintSemicolon
\SetKwInput{KwInput}{Input}               
\SetKwInput{KwOutput}{Output} 
\KwInput{adjacency matrix $\bA$}
\KwOutput{label vector $\hat{\bx}$}       
compute $\bm{u}_2$, the eigenvector of $\bA$ associated with the second largest eigenvalue of $\bA$ \label{line:eig} \\
output $\hat{\bm x} = \mbox{sgn}(\bm{u}_2)$
\caption{Vanilla Spectral Method}
\label{alg:spectral}
\end{algorithm}

By combining the results in~\cite{abbe2017entrywise} with those in Sections~\ref{sec:oim} and~\ref{sec:ppm}, we now show that if the vanilla spectral method is implemented as $\mO(\log n/\log \log n)$ orthogonal iterations followed by a single projection onto $\mF$, then it can exactly recover the underlying communities with high probability. In particular, such an implementation runs in $\mO(n\log^2n/\log\log n)$ time, which, to the best of our knowledge, is currently the best complexity bound for the vanilla spectral method in the context of exact community recovery in the binary symmetric SBM. To begin, we recall~\cite[Corollary 3.1]{abbe2017entrywise}, which states that with probability $1-\mO(n^{-3})$, one has
\begin{equation} \label{eq:entrywise-bd}
\min_{\theta\in\{\pm1\}} \left\| \bm{u}_2 - \theta\frac{\bA\bxs}{\nu_2\sqrt{n}} \right\|_{\infty} \le \frac{C}{\sqrt{n}\log\log n}, 
\end{equation}
where $C>0$ is a constant whose value depends only on $\alpha$ and $\beta$, and $\nu_2$ is the second-largest eigenvalue of $\E[\bA]$ (see~\eqref{eq:nu-eig-EA}).
By Theorem~\ref{thm:OI}, after 
\[ N = \left\lceil  \frac{3\log n + 2 \log\log\log n + 2\log\left(\frac{16\alpha+18\beta}{C\beta}\right)}{\log\log n + 2\log\left(\frac{\alpha-\beta}{4c_1}\right)} \right\rceil \ge \bar{K} \]
orthogonal iterations, where $\bar{K}$ is given in~\eqref{eq:bK}, we obtain a vector $\underline{\bq}^{(N)}$ satisfying
\begin{equation} \label{eq:oi-bd}
\min_{\theta\in\{\pm1\}} \| \underline{\bq}^{(N)} - \theta \bm{u}_2 \|_2 \le \frac{C}{\sqrt{n}\log\log n}.
\end{equation}
From~\eqref{eq:entrywise-bd} and~\eqref{eq:oi-bd}, we deduce the existence of $\tilde{\theta}\in\{\pm1\}$ such that 
\begin{equation} \label{eq:approx-bd}
\left\| \underline{\bq}^{(N)} - \tilde{\theta}\frac{\bA\bxs}{\nu_2\sqrt{n}} \right\|_{\infty} \le \frac{2C}{\sqrt{n}\log\log n}.
\end{equation}
Now, by~\eqref{eq:opt-tail-bd-B}, it holds with probability at least $1-n^{-c_5}$ that
\begin{equation} \label{eq:scaled-gt}
\min_{i \in [n]} \left| \frac{(\bA\bxs)_i}{\nu_2\sqrt{n}} \right| \geq \frac{2\gamma}{\sqrt{n}(\alpha-\beta)}.
\end{equation}
It follows that
\[ 
\bxs = \mbox{sgn}(\bxs) = \mbox{sgn}(\bA\bxs) = \mbox{sgn}\left( \frac{\bA\bxs}{\nu_2\sqrt{n}} \right) = \tilde{\theta} \cdot \mbox{sgn}(\underline{\bq}^{(N)}), 
\]
where the second equality is due to~\eqref{eq:opt-tail-bd-B} and the fourth equality uses~\eqref{eq:approx-bd} and~\eqref{eq:scaled-gt}. The above implies that $\underline{\bq}^{(N)}$ has $n/2$ positive entries and $n/2$ negative entries. Thus, by Proposition~\ref{prop:cf-proj}, we have $\mP(\underline{\bq}^{(N)}) \in \{\pm\bxs\}$ as desired.

\section{Numerical Results}\label{sec:num}

In this section, we report the recovery performance and numerical efficiency of our proposed two-stage method (which we denote by PPM in this section for ease of reference) for community recovery on both synthetic and real data sets.  We also compare our approach with four existing approaches, which are the SDP-based approach in \cite{amini2018semidefinite}, the manifold optimization (MFO)-based approach in \cite{bandeira2016low}, the spectral clustering (SC) approach in \cite{abbe2017entrywise}, and the two-stage approach based on the generalized power method (GPM) in \cite{wang2020non}. In the implementation, we use alternating direction method of multipliers (ADMM) to solve the SDP as suggested in \cite{amini2018semidefinite}, manifold gradient descent (MGD) method to solve the MFO, and the MATLAB function \textsf{eigs} for computing the eigenvector that is needed in the SC approach. 
Our codes are implemented in MATLAB R2020a and can be downloaded at \url{https://github.com/peng8wang/MP-Exact-Recovery-in-SBM}. All the experiments are conducted on a PC with 16GB memory and Intel(R) Core(TM) i5-8600 3.10GHz CPU. 

\subsection{Phase Transition and Computation Time}

We first examine the phase transition property and runtime of the aforementioned methods for recovering communities in graphs that are randomly generated according to the binary symmetric SBM, both with and without self-loops (see Definition~\ref{model:SBM} and Section~\ref{sec:no-loop}). We choose $n=300$ in the experiments and let the parameters $\alpha$ and $\beta$ in~\eqref{eq:p-q} vary from $0$ to $30$ and $0$ to $10$ with increments of $0.5$ and $0.4$, respectively. For every pair of $\alpha$ and $\beta$, we generate $40$ instances and calculate, for all the methods, the ratio of exact recovery. The simulation results are presented in Figure \ref{fig-1}, Figure \ref{fig-6},  and Table \ref{table-1}. It can be seen that all methods exhibit a phase transition phenomenon and the recovery performance of PPM is slightly better than the other three methods. Moreover, Figures \ref{fig-1}(a) and \ref{fig-6}(a) suggest that PPM can achieve the optimal recovery threshold on graphs with and without self-loops, respectively. This supports the results in Theorem \ref{thm:main} and Section~\ref{sec:no-loop}. In Table \ref{table-1}, we record the total CPU time consumed by each approach for completing the phase transition experiment.  It can be observed from the table that PPM is comparable to GPM, slightly better than SC, and substantially faster than SDP and MGD.

\begin{figure*}[!htbp]
	\begin{minipage}[b]{0.32\linewidth}
		\centerline{\includegraphics[width=\linewidth]{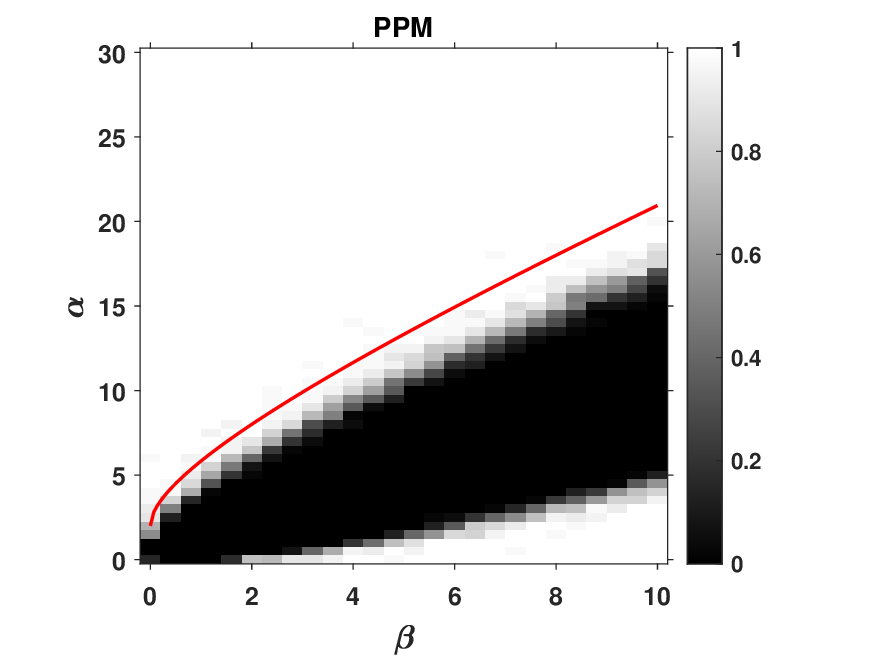}}
		\centerline{(a) PPM} \smallskip
	\end{minipage}
    \begin{minipage}[b]{0.32\linewidth}
		\centerline{\includegraphics[width=\linewidth]{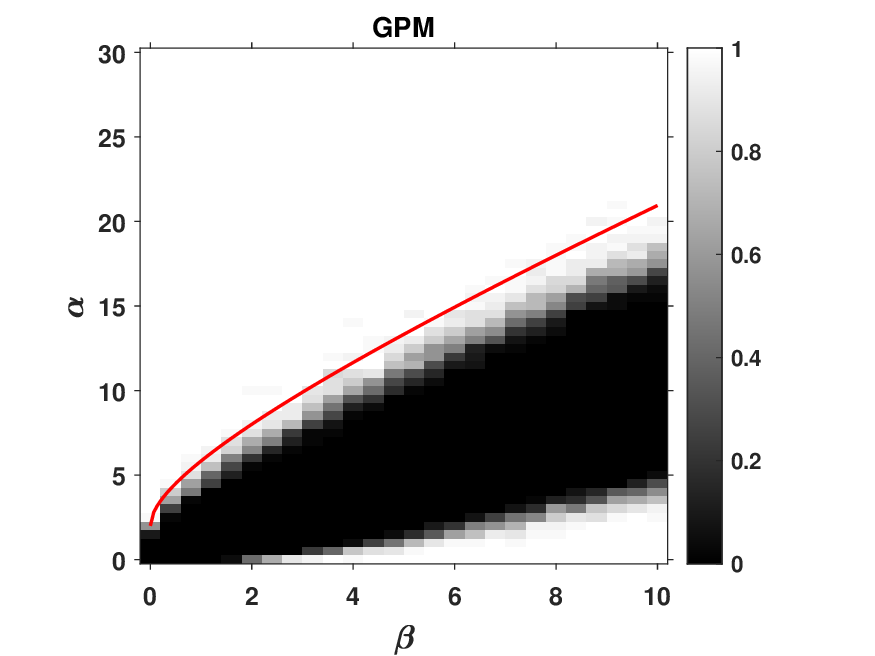}}
		\centerline{(b) GPM} \smallskip
	\end{minipage}
	\begin{minipage}[b]{0.32\linewidth}
		\centerline{\includegraphics[width=\linewidth]{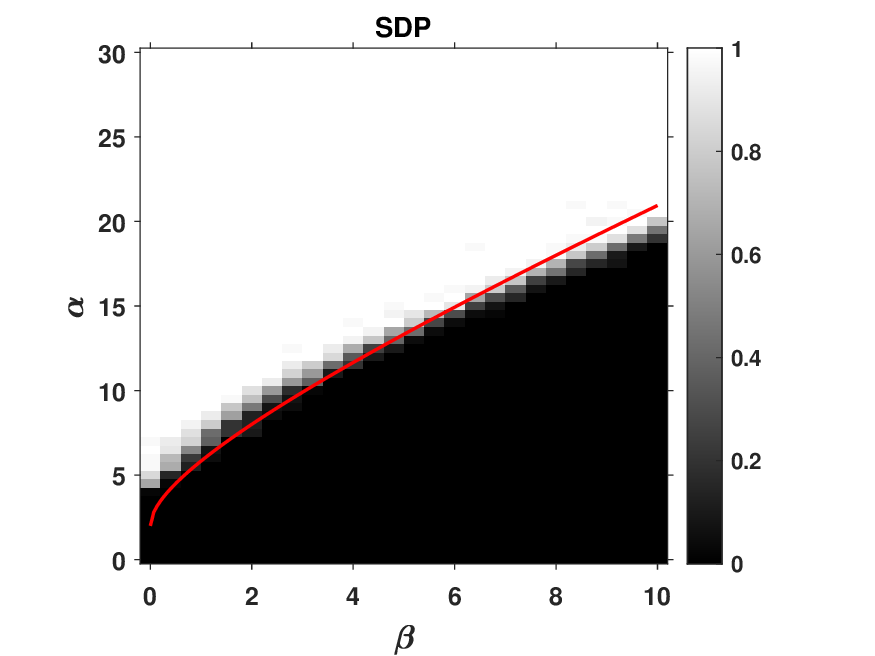}}
		\centerline{(c) SDP}\smallskip
	\end{minipage}
	\begin{center}
	\begin{minipage}[b]{0.32\linewidth}
		\centerline{\includegraphics[width=\linewidth]{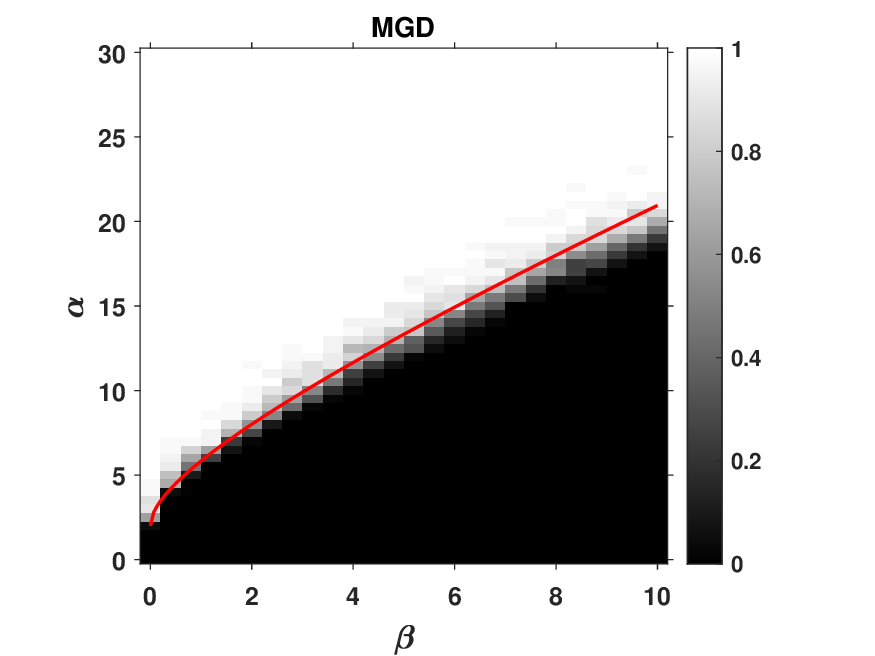}}
		\centerline{(d) MGD}\smallskip
	\end{minipage}
	\begin{minipage}[b]{0.32\linewidth}
		\centerline{\includegraphics[width=\linewidth]{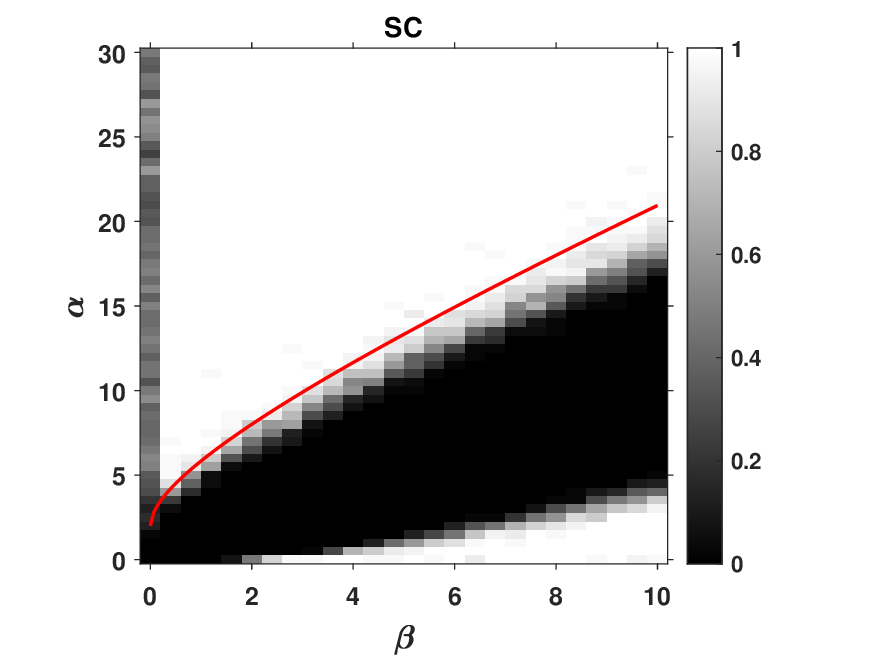}}
		\centerline{(e) SC}\smallskip
	\end{minipage}
	\end{center}
	\caption{Phase transition in graphs generated by the SBM with self-loops: The $x$-axis is $\beta$, which ranges from $0$ to $10$ with an increment of $2$; the $y$-axis is $\alpha$, which ranges from $0$ to $30$ with an increment of $5$. Darker pixels represent lower empirical probability of success. The red curve is the information-theoretic threshold $\sqrt{\alpha}-\sqrt{\beta}=\sqrt{2}$.}
	\label{fig-1}
\end{figure*}

\begin{figure*}[!htbp]
	\begin{minipage}[b]{0.32\linewidth}
		\centerline{\includegraphics[width=\linewidth]{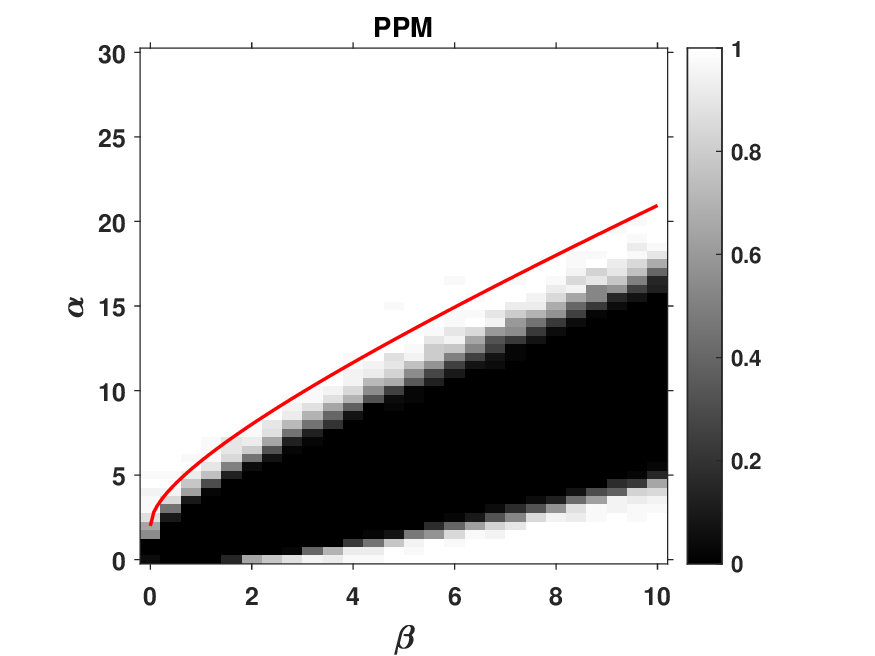}}
		\centerline{(a) PPM} \smallskip
	\end{minipage}
	\begin{minipage}[b]{0.32\linewidth}
		\centerline{\includegraphics[width=\linewidth]{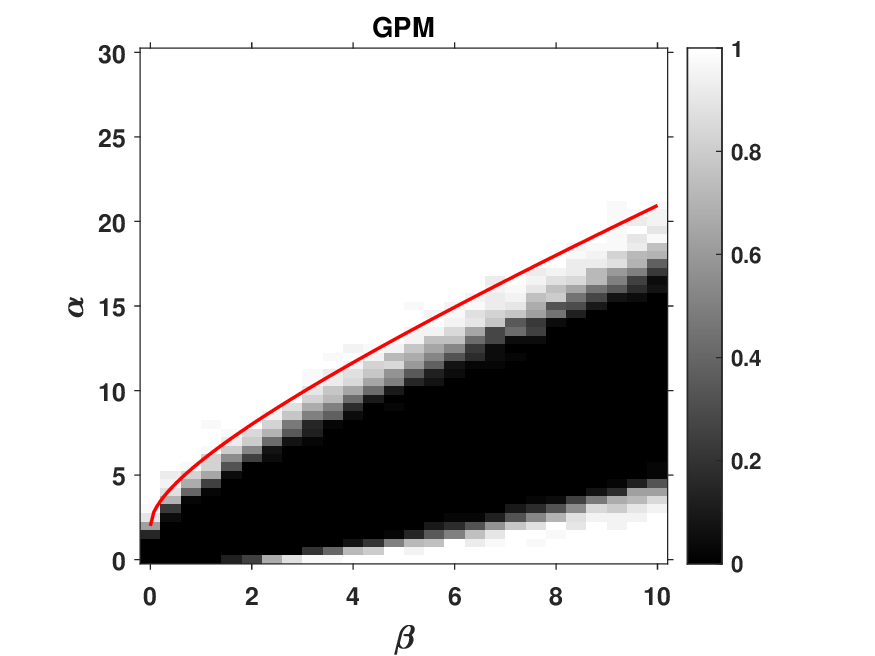}}
		\centerline{(b) GPM} \smallskip
	\end{minipage}
	\begin{minipage}[b]{0.32\linewidth}
		\centerline{\includegraphics[width=\linewidth]{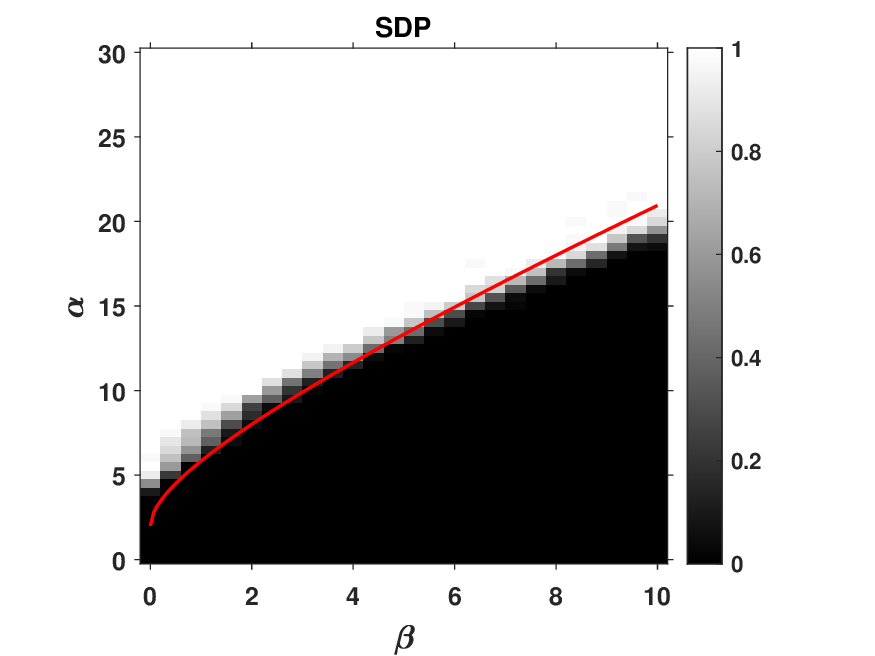}}
		\centerline{(c) SDP}\smallskip
	\end{minipage}
	\begin{center}
	\begin{minipage}[b]{0.32\linewidth}
		\centerline{\includegraphics[width=\linewidth]{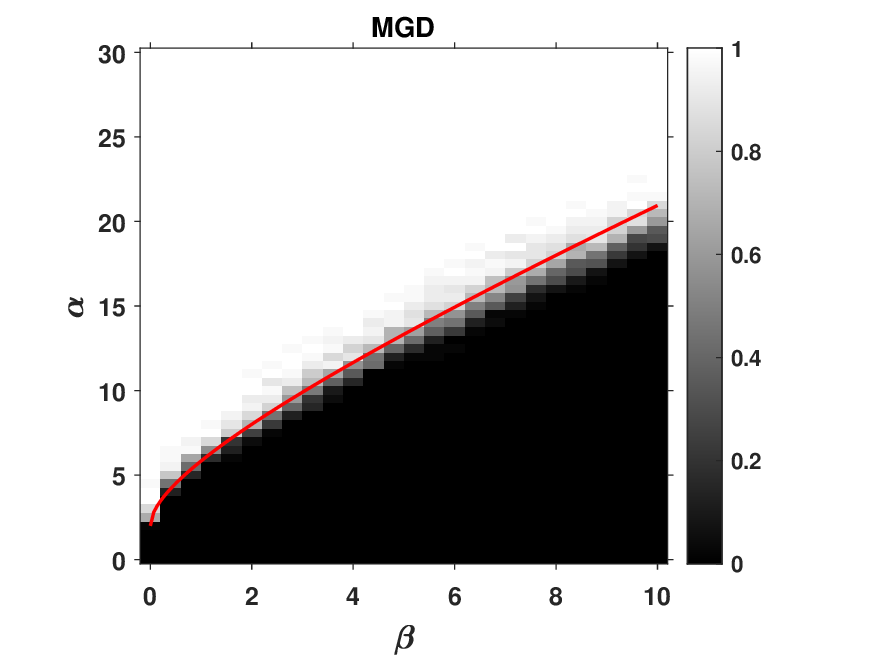}}
		\centerline{(d) MGD}\smallskip
	\end{minipage}
	\begin{minipage}[b]{0.32\linewidth}
		\centerline{\includegraphics[width=\linewidth]{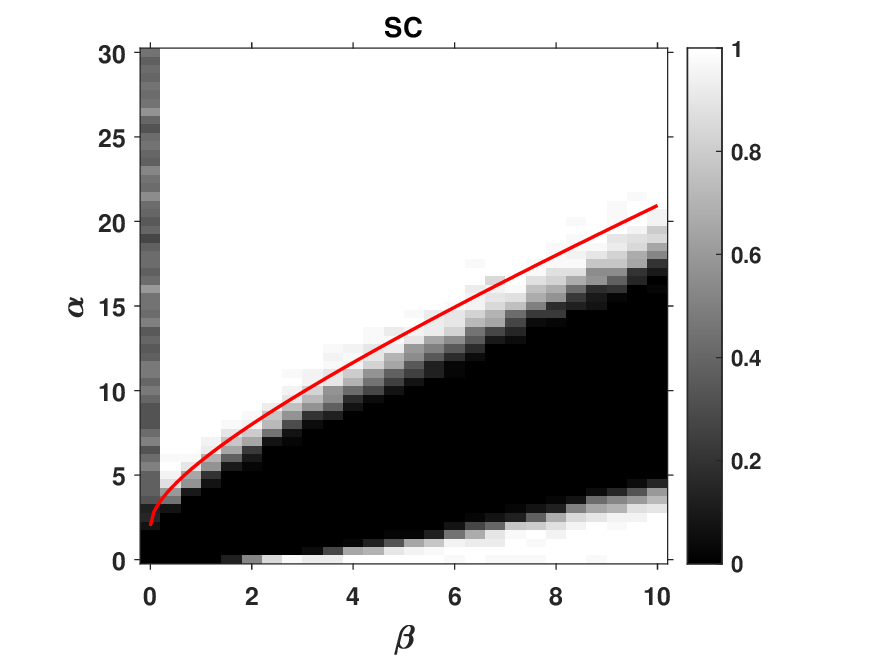}}
		\centerline{(e) SC}\smallskip
	\end{minipage}
	\end{center}
	\caption{Phase transition in graphs generated by the SBM without self-loops: The $x$-axis is $\beta$, which ranges from $0$ to $10$ with an increment of $2$; the $y$-axis is $\alpha$, which ranges from $0$ to $30$ with an increment of $5$. Darker pixels represent lower empirical probability of success. The red curve is the information-theoretic threshold $\sqrt{\alpha}-\sqrt{\beta}=\sqrt{2}$.}
	\label{fig-6}
\end{figure*}

\begin{table}[!htbp]
\caption{Phase transition: Total CPU time (in seconds) of the different approaches.}
\label{table-1}
\vskip 0.15in
\begin{center}
\begin{small}
\begin{tabular}{cccccc}
\hline
Methods & PPM & GPM & SDP &  MGD & SC\\
\hline
Self-loops & 18 &  {\bf 14} & 8195 & 922 & 104\\
No self-loop &  19 & {\bf 14} & 8356 & 932 & 103 \\
\hline
\end{tabular}
\end{small}
\end{center}
\end{table}

\vspace{-0.5cm}

\subsection{Convergence Performance}

Next, we study the convergence performance of PPM and MGD, both of which have similar per-iteration cost, and report the number of iterations needed to exactly identify the two communities in graphs generated according to the binary symmetric SBM. We do not report the performance of GPM, SDP, and SC, as GPM and PPM have very similar performance, SDP cannot be solved to high accuracy using ADMM, and SC can be directly solved using the MATLAB function \textsf{eigs}. We conduct 3 sets of numerical tests each on graphs with and without self-loops, which correspond to $\beta \in \{4,8,16\}$. In each set, we generate 5 graphs of dimension $n=2000$, which correspond to $\alpha=(\sqrt{\beta}+\sqrt{2})^2+i$ for $i \in \{1,2,3,4,5\}$. Such a setting ensures that the information-theoretic threshold for exact recovery is met. Let $\bx^k$ and $\bm{Q}^k$ denote $k$-th iterate of PPM and MGD, respectively. In Figures \ref{fig-2} and \ref{fig-5}, we plot the distances of the iterates to the ground truth $\|\bx^k\bx^{k^T}-\bx^*\bx^{*^T}\|_F$ and $\|\bm{Q}^k\bm{Q}^{k^T}-\bx^*\bx^{*^T}\|_F$ against the iteration number for PPM and MGD, respectively. It can be observed that PPM exhibits a finite termination phenomenon and converges to the ground truth much faster than MGD in graphs both with and without self-loops. This also corroborates the iteration complexity established in Theorem~\ref{thm:main} and Section~\ref{sec:no-loop}.

\begin{figure*}[!htbp]
\begin{center}
	\begin{minipage}[b]{0.45\linewidth}
		\centering
		\centerline{\includegraphics[width=\linewidth]{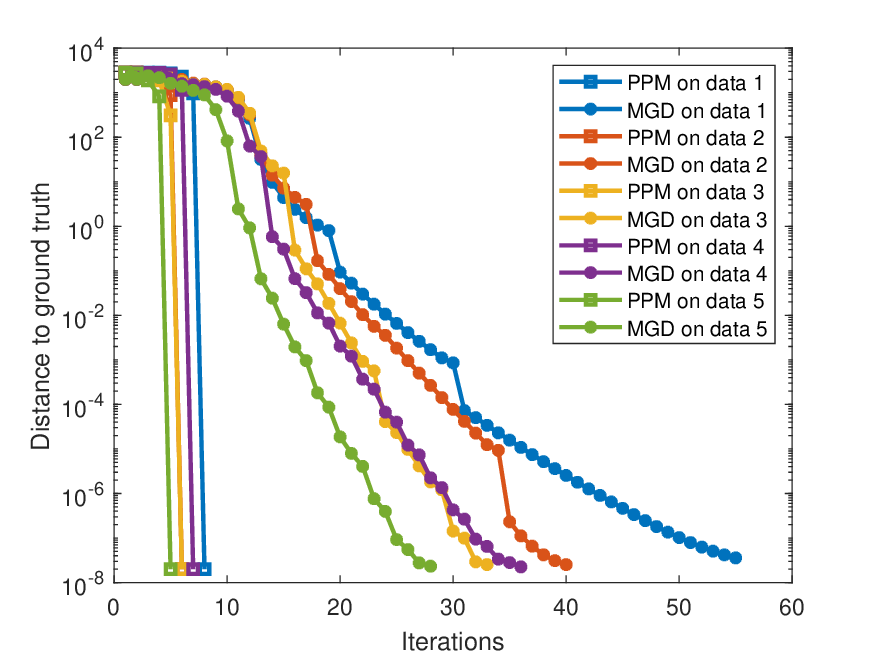}}
		\centerline{(a) $\beta=4$}\smallskip
	\end{minipage}
	\begin{minipage}[b]{0.45\linewidth}
		\centering
		\centerline{\includegraphics[width=\linewidth]{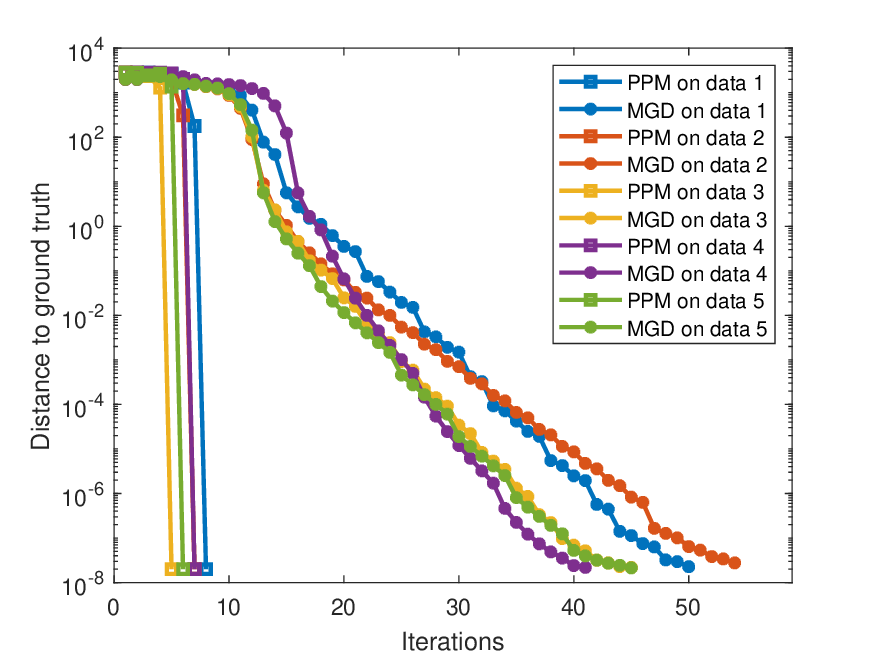}}
		\centerline{(b) $\beta=8$}\smallskip
	\end{minipage}
	\begin{minipage}[b]{0.45\linewidth}
		\centering
		\centerline{\includegraphics[width=\linewidth]{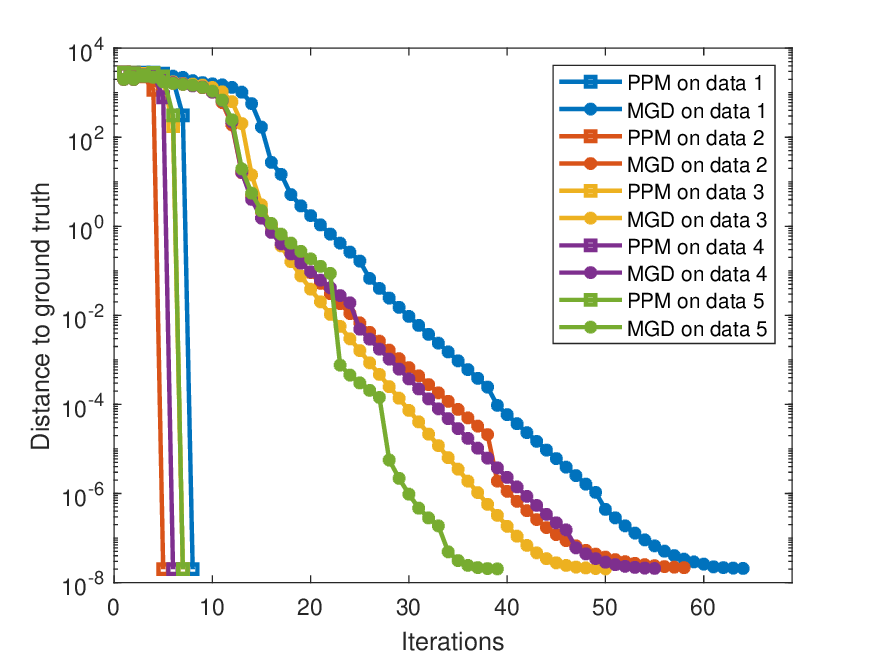}}
		\centerline{(c) $\beta=16$}\smallskip
	\end{minipage}
\end{center}
	\caption{Convergence performance on graphs generated by the SBM with self-loops: The $x$-axis is number of iterations, the $y$-axis is distance to ground truth, which is given by $\|\bx^k\bx^{k^T} - \bxs\bx^{*^T}\|_F$ for PPM and $\|\bm{Q}^k\bm{Q}^{k^T} - \bxs\bx^{*^T}\|_F$ for MGD. Here, $\bx^k$ and $\bm{Q}^k$ are the $k$-th iterates generated by the PPM and the MGD, respectively.}
	\label{fig-2}
\end{figure*}
\begin{figure*}[!htbp]
\begin{center}
	\begin{minipage}[b]{0.45\linewidth}
		\centering
		\centerline{\includegraphics[width=\linewidth]{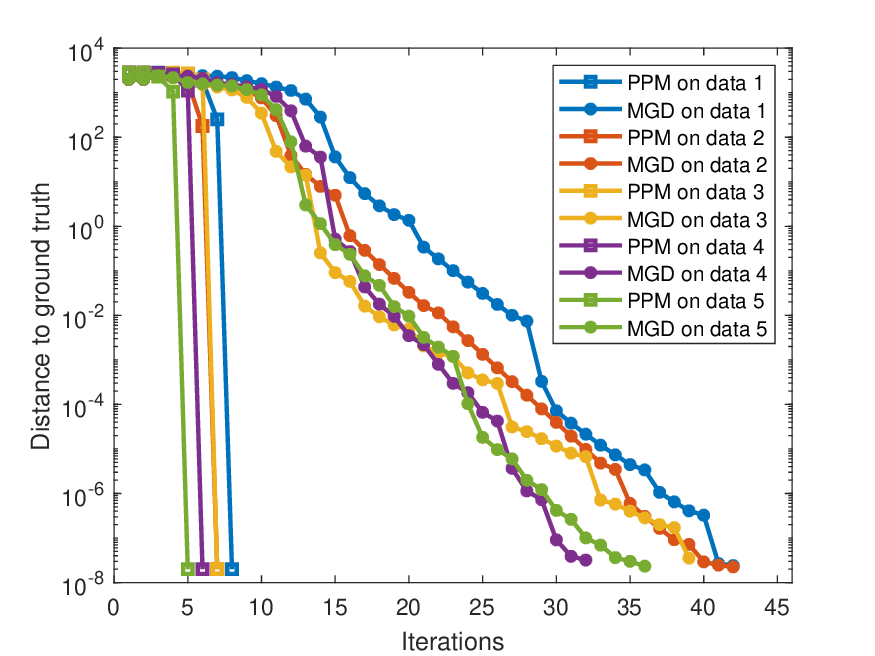}}
		\centerline{(a) $\beta=4$}\smallskip
	\end{minipage}
	\begin{minipage}[b]{0.45\linewidth}
		\centering
		\centerline{\includegraphics[width=\linewidth]{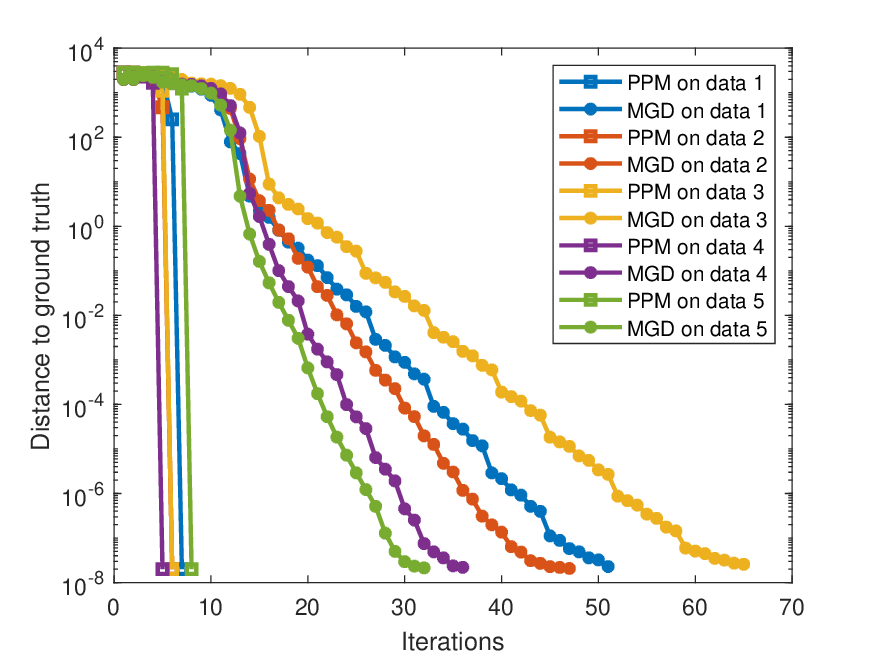}}
		\centerline{(b) $\beta=8$}\smallskip
	\end{minipage}
	\begin{minipage}[b]{0.45\linewidth}
		\centering
		\centerline{\includegraphics[width=\linewidth]{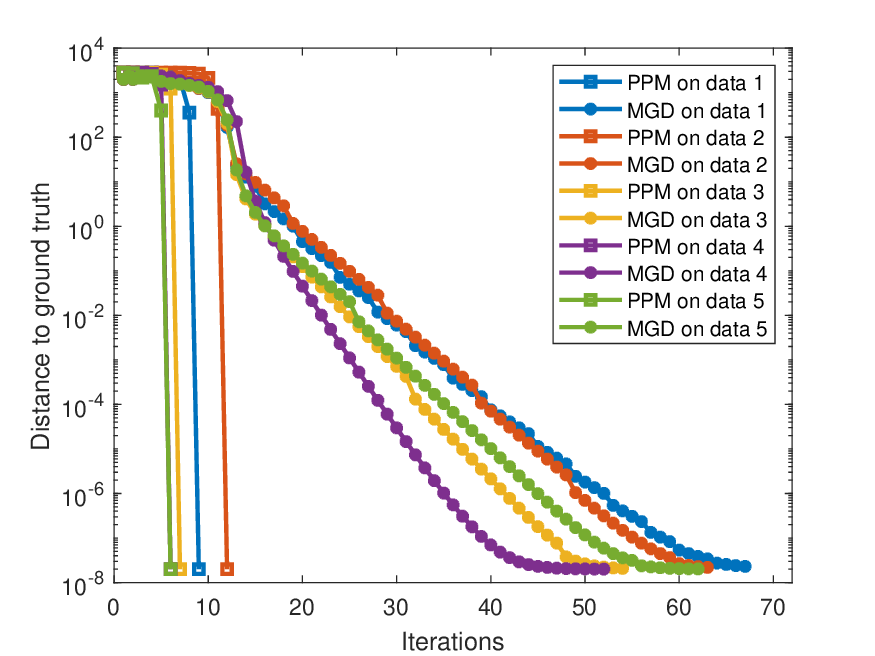}}
		\centerline{(c) $\beta=16$}\smallskip
	\end{minipage}
\end{center}
	\caption{Convergence performance on graphs generated by the SBM without self-loops: The $x$-axis is number of iterations, the $y$-axis is distance to ground truth, which is given by $\|\bx^k\bx^{k^T} - \bxs\bx^{*^T}\|_F$ for PPM and $\|\bm{Q}^k\bm{Q}^{k^T} - \bxs\bx^{*^T}\|_F$ for MGD. Here, $\bx^k$ and $\bm{Q}^k$ are the $k$-th iterates generated by the PPM and the MGD, respectively.}
	\label{fig-5}
\end{figure*}

\vspace{-0.5cm}
\subsection{Computational Efficiency}

In this sub-section, we compare the computational efficiency of our proposed method with GPM, MGD, SDP, and SC on both synthetic and real data sets. For the synthetic data sets, we fix $\beta=16$, $\alpha=(\sqrt{\beta}+\sqrt{2})^2+1$, and generate three graphs of dimension $n = 2000$, $10000$, and $20000$, respectively. For the real ones, we use the data sets \emph{polbooks} and \emph{polblogs} downloaded from UF Sparse Matrix Collection \cite{davis2011university}.\footnote{\url{https://sparse.tamu.edu/}}  Since these real-world networks have unbalanced or multiple communities, we extract 2 balanced communities from them. The sizes of each community extracted from \emph{polbooks} and \emph{polblogs} are 43 and 732, respectively. The stopping criteria for the tested algorithms are set as follows. For PPM, we terminate it when $\| \bx^k - \by^k\|_2 < 10^{-3}$ for some $\by^k \in \mP(\bx^k+\bA\bx^k)$; for GPM, we terminate it when $\| \bx^k - \by^k\|_2 < 10^{-3}$ for some $\by^k \in \mbox{sgn}\left(\bx^k+\bA\bx^k-\bo_n^T\bA\bo_n/n^2\cdot \bo_n^T\bx^k\bo_n\right)$; for MGD, we terminate it when the norm of the manifold gradient is less then $10^{-3}$; for ADMM, we terminate it when the norm of the difference of two consecutive iterates is less than $10^{-3}$. No stopping criterion is needed for SC as it simply employs the MATLAB \textsf{eigs} function with some post-processing. We run each algorithm 10 times from randomly generated initial points and select the best solution (in terms of function value) as its recovery solution. Moreover, we set the maximum iteration number as 2000 for every algorithm. To compare the computational efficiency of the tested algorithms, we record their CPU time, averaged over 10 runs, and present the results in Table~\ref{tab-1}. It can be observed that our proposed method is nearly as fast as GPM, slightly better than MGD and SC, and substantially faster than SDP. 

All the tested methods can achieve exact recovery on synthetic data sets. Their recovery performance on the two real data sets \emph{polbooks} and \emph{polblogs} are presented in Figures \ref{fig-3} and \ref{fig-4}, respectively. According to the ground truth of \emph{polbooks}, the number of misclassified vertices by PPM, GPM, SDP, MGD, and SC are 0, 4, 1, 4, and 3, respectively. As for \emph{polblogs}, the number of misclassified vertices by PPM, GPM, SDP, MGD, and SC are 64, 698, 289, 294, and 194, respectively. These results demonstrate that our proposed method is comparable to SDP, MGD, SC and is better than GPM in terms of recovery performance on the two real data sets.  

\begin{table}[!htbp]
	\caption{CPU times (in seconds) of the algorithms on synthetic and real data sets.}
	\label{tab-1}
	\begin{center}
	\begin{threeparttable}
		\begin{tabular}{ c c c c c c}
			\hline
			& PPM & GPM & SDP & MGD & SC   \\ \hline
			$n=2000$ &  0.005 & {\bf 0.004} & 79.47 & 0.313 & 0.014 \\
			$n=10000$ & 0.031 & {\bf 0.028}  & --\tnote{*}	& 2.562 & 0.113  \\
			$n=20000$ & {\bf 0.074} & {\bf 0.074} &	--	& 4.687	& 0.268 \\ 
			\emph{polbooks} &  0.003 & {\bf 0.002} & 0.172 & 0.059 & 0.098 \\ 
			\emph{polblogs} & {\bf 0.019} & 0.102 & 492.1 & 6.118 & 0.021 \\  \hline 			
		\end{tabular}
		\begin{tablenotes}
 			\item[*] ``--'' denotes out of memory.
        \end{tablenotes}
   	\end{threeparttable}
	\end{center}
\end{table}

\begin{figure*}[t]
	\centering
	\begin{minipage}[b]{0.325\linewidth}
		\centering
		\centerline{\includegraphics[width=\linewidth]{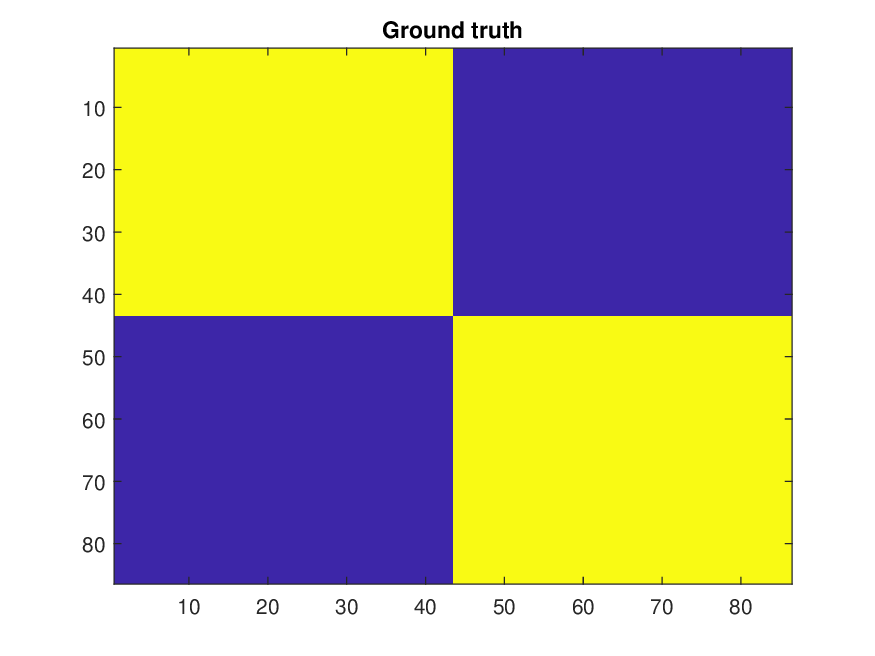}}
		\centerline{(a) Ground Truth}\medskip
	\end{minipage}
	\hfill
	\begin{minipage}[b]{0.325\linewidth}
		\centering
		\centerline{\includegraphics[width=\linewidth]{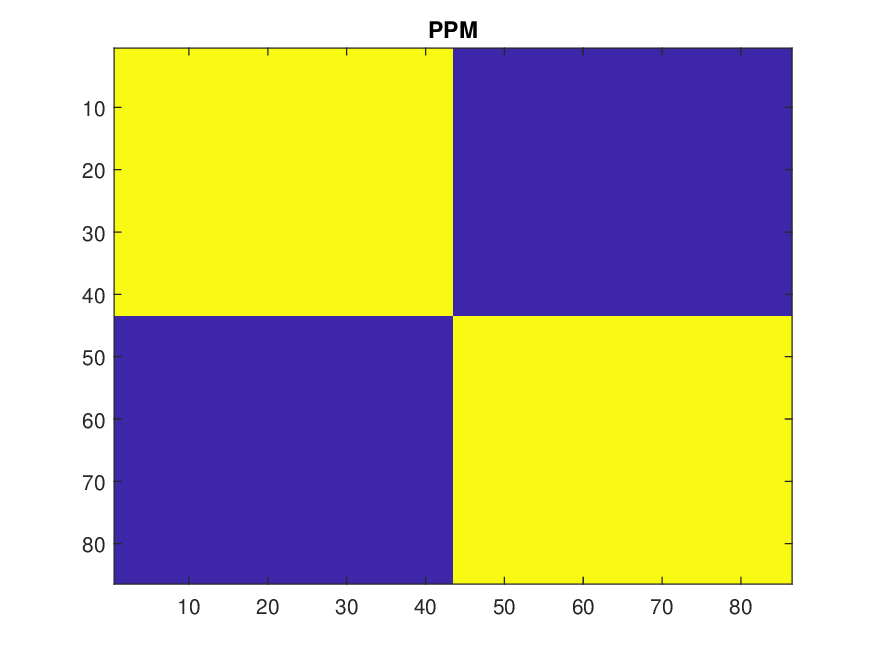}}
		\centerline{(b) PPM}\medskip
	\end{minipage}
	\hfill
	\begin{minipage}[b]{0.325\linewidth}
		\centering
		\centerline{\includegraphics[width=\linewidth]{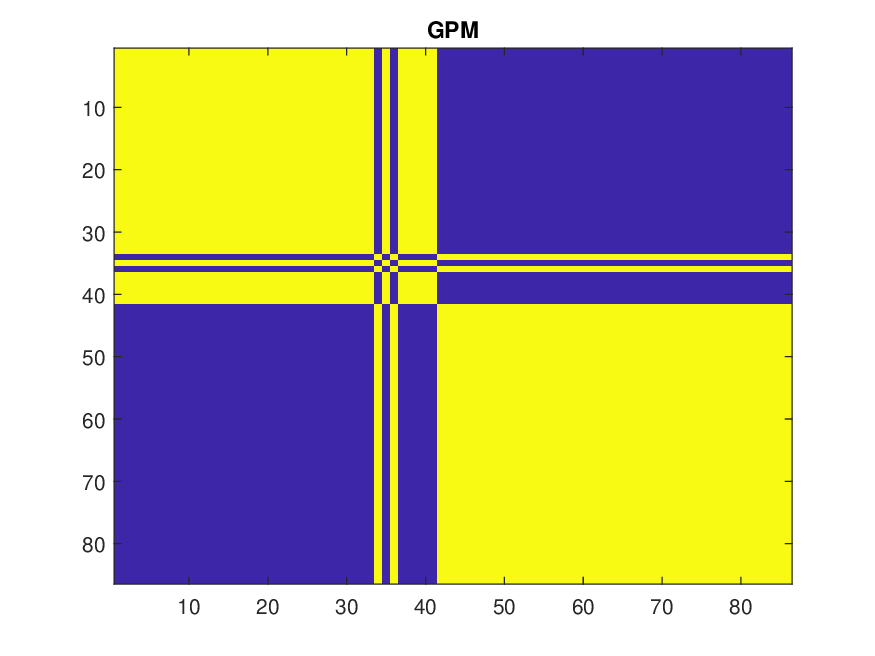}}
		\centerline{(c) GPM}\medskip
	\end{minipage}
	\hfill
	\begin{minipage}[b]{0.325\linewidth}
		\centering
		\centerline{\includegraphics[width=\linewidth]{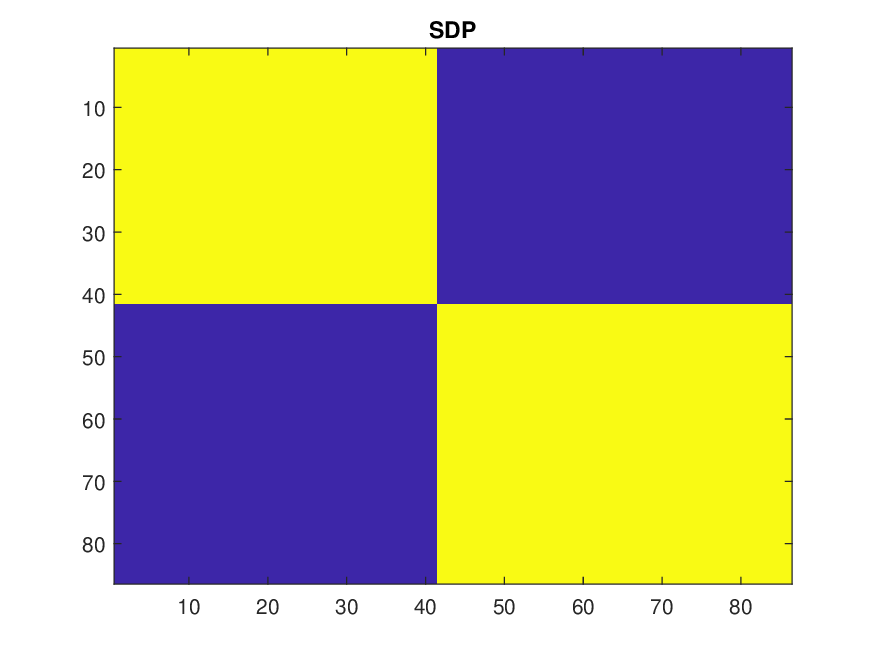}}
		\centerline{(d) SDP}\medskip
	\end{minipage}
		\hfill
	\begin{minipage}[b]{0.325\linewidth}
		\centering
		\centerline{\includegraphics[width=\linewidth]{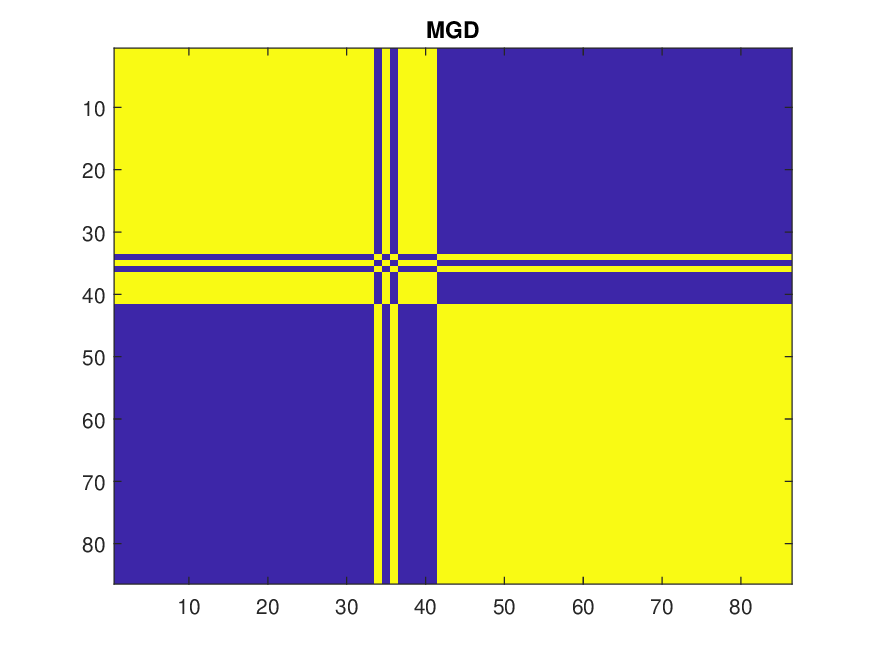}}
		\centerline{(e) MGD}\medskip
	\end{minipage}
	\hfill
	\begin{minipage}[b]{0.325\linewidth}
		\centering
		\centerline{\includegraphics[width=\linewidth]{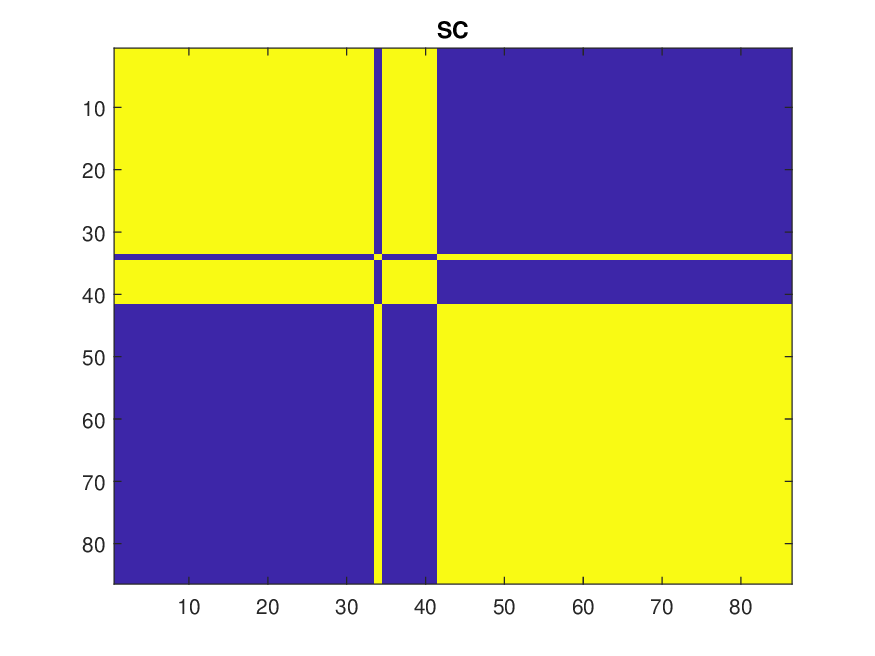}}
		\centerline{(f) SC}\medskip
	\end{minipage}
	\vskip -0.1in
	\caption{Recovery performance on the \emph{polbooks} network: The $x$-axis and $y$-axis give the labels of the vertices (from vertex $1$ to vertex $86$) in the network \emph{polbooks}.}
	\label{fig-3}
\end{figure*}

\begin{figure*}[t]
	\centering
	\begin{minipage}[b]{0.325\linewidth}
		\centering
		\centerline{\includegraphics[width=\linewidth]{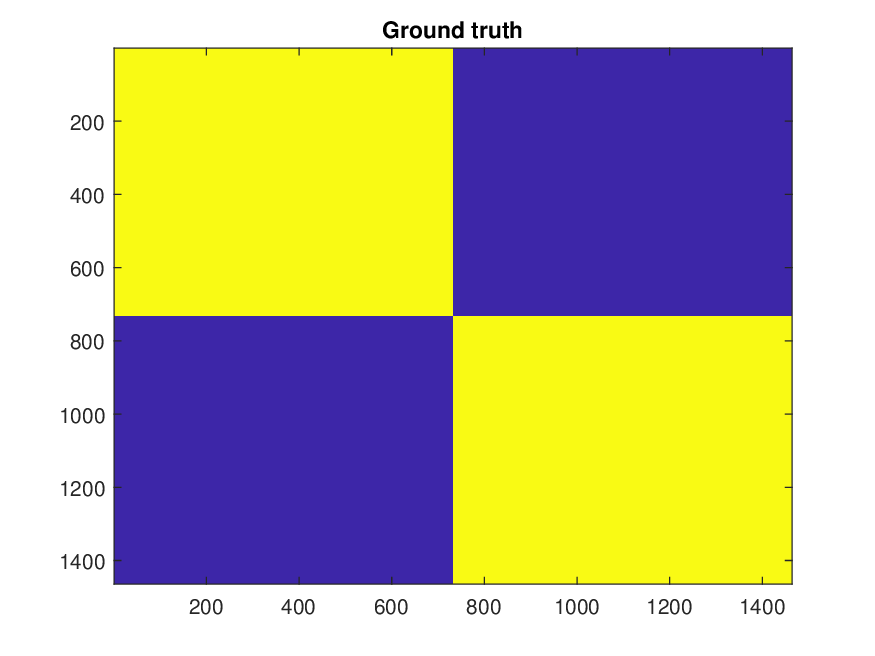}}
		\centerline{(a) Ground truth}\medskip
	\end{minipage}
	\hfill		
	\begin{minipage}[b]{0.325\linewidth}
		\centering
		\centerline{\includegraphics[width=\linewidth]{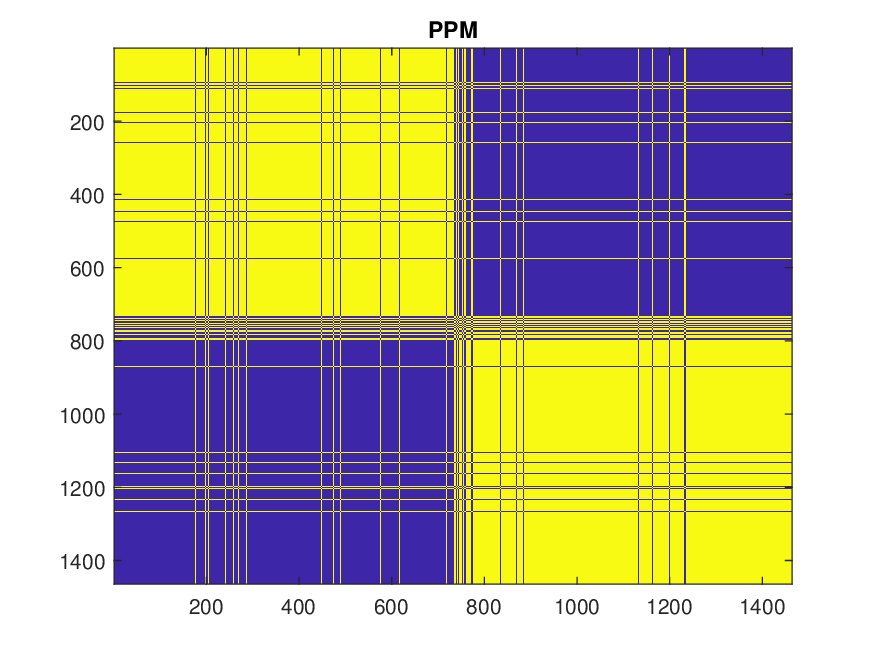}}
		\centerline{(b) PPM}\medskip
	\end{minipage}
	\hfill
	\begin{minipage}[b]{0.325\linewidth}
		\centering
		\centerline{\includegraphics[width=\linewidth]{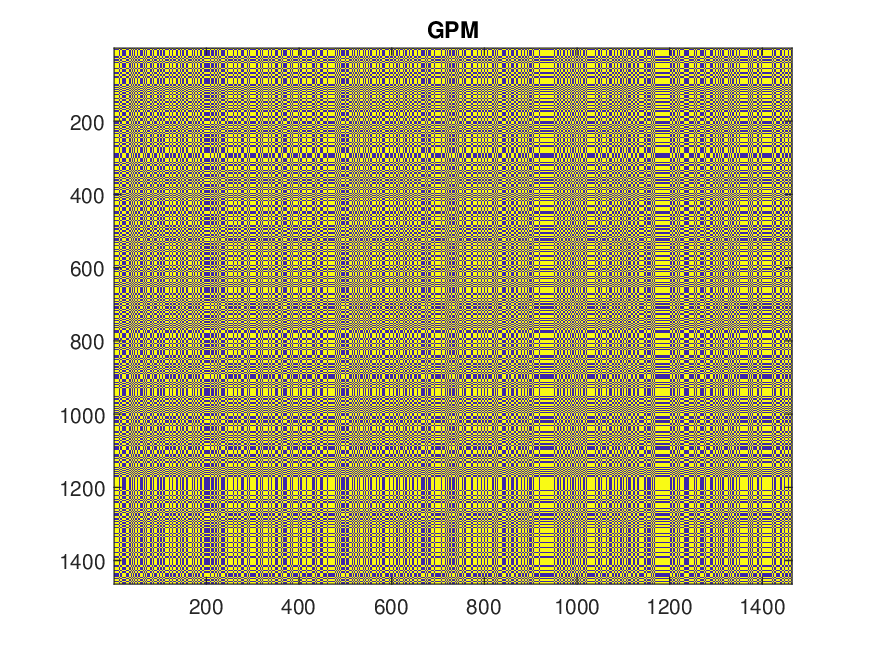}}
		\centerline{(c) GPM}\medskip
	\end{minipage}
	\hfill
	\begin{minipage}[b]{0.325\linewidth}
		\centering
		\centerline{\includegraphics[width=\linewidth]{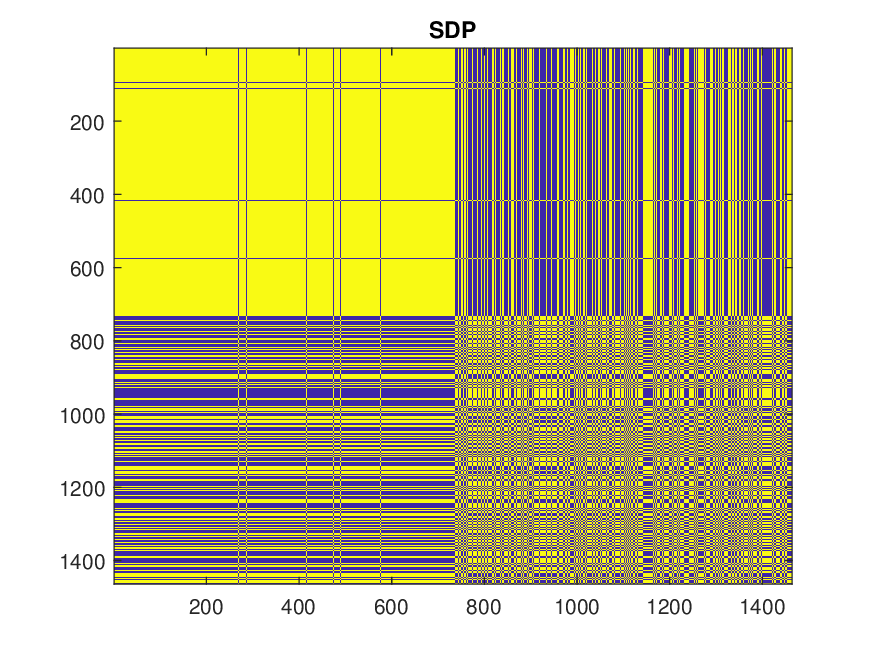}}
		\centerline{(d) SDP}\medskip
	\end{minipage}
	\hfill
	\begin{minipage}[b]{0.325\linewidth}
		\centering
		\centerline{\includegraphics[width=\linewidth]{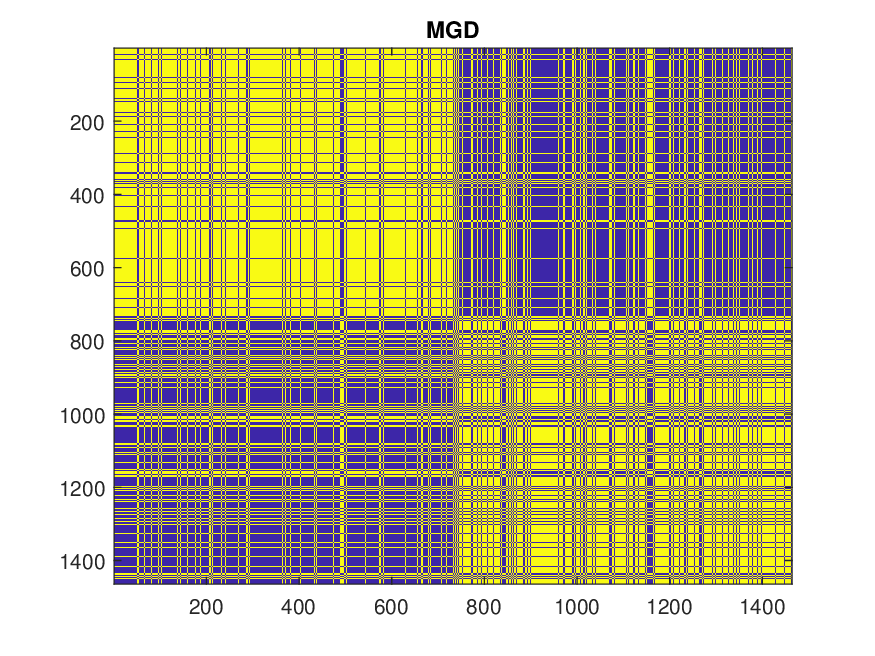}}
		\centerline{(e) MGD}\medskip
	\end{minipage}
	\hfill
	\begin{minipage}[b]{0.325\linewidth}
		\centering
		\centerline{\includegraphics[width=\linewidth]{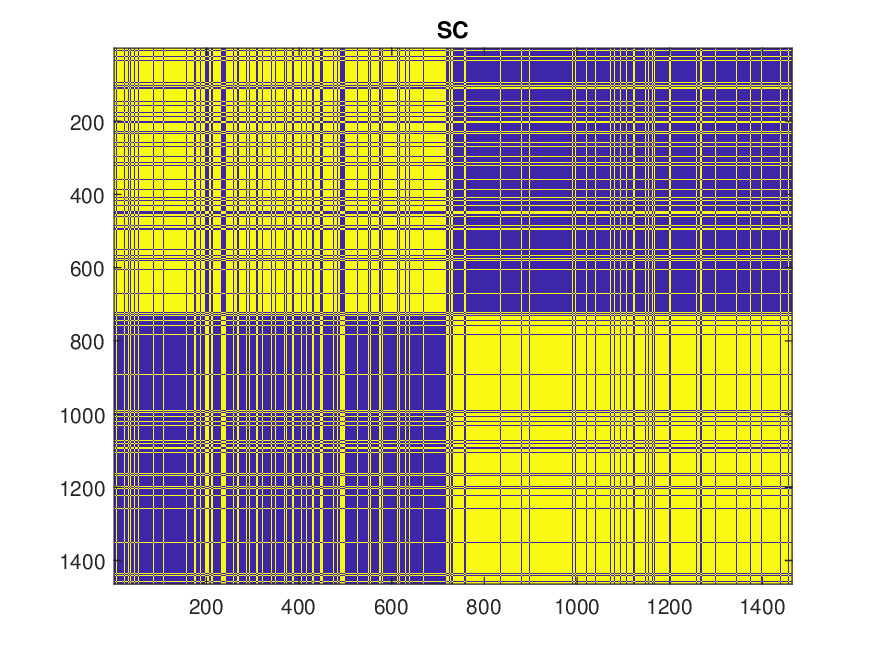}}
		\centerline{(f) SC}\medskip
	\end{minipage}
	\vskip -0.1in
	\caption{Recovery performance on the \emph{polblogs} network: The $x$-axis and $y$-axis give the labels of the vertices (from vertex $1$ to vertex $1464$) in the network \emph{polblogs}.}
	\label{fig-4}
\end{figure*}

\section{Conclusions} \label{sec:concl}

In this work, we proposed a two-stage iterative algorithm that provably achieves exact recovery down to the information-theoretic limit in the binary symmetric SBM and has a runtime of $\mO(n\log^2n/\log\log n)$. The complexity bound is among the best for algorithms that have the same recovery performance. In the process of establishing our main results, we developed new analyses of the orthogonal iterations and projected power iterations used in our algorithm, which could be of independent interest. Our numerical results on synthetic and real data sets demonstrate the strong recovery performance and high computational efficiency of the proposed algorithm. A natural future direction is to extend the proposed approach to tackle recovery tasks in more general SBMs.


\bibliographystyle{spmpsci}      
\bibliography{stochastic-block-model}   


\end{document}